%% file: main.tex
\documentclass[a4paper,10pt,final]{amsart}
\usepackage[utf8]{inputenc}
\usepackage{fouriernc, amsmath, amsthm, amsfonts, amssymb,epsfig,enumerate, wrapfig, scalerel, tikz,  color, accents, fancybox, rotating, comment, afterpage, changepage, colortbl,  pdfpages, datetime, fancyhdr, float, xlop, blkarray, stmaryrd, hyperref, bookmark, cleveref, setspace, ifthen, tikz-3dplot, pdflscape, leftindex}
\title{A foamy approach to Soergel bimodules} \author{Mikhail Khovanov} 
 \address{Department of Mathematics, Johns Hopkins University, Baltimore, MD 21218, USA}
 \email{\href{mailto:khovanov@math.columbia.edu}{khovanov@jhu.edu}}
 \author{Louis-Hadrien Robert}
 \address{Université Clermont Auvergne, LMBP, Campus des
   Cézeaux, 3 place Vasarely, TSA 60026 -- CS 60026, 63178 Aubière Cedex, France}
 \email{\href{mailto:louis-hadrien.robert@uca.fr}{louis-hadrien.robert@uca.fr}}
 \author{Emmanuel Wagner}
 \address{Univ Paris Cit\'e, IMJ-PRG, Univ Paris Sorbonne, UMR 7586 CNRS, F-75013, Paris,
France}
 \email{\href{mailto:wagner@imj-prg.fr}{wagner@imj-prg.fr}}
 \hypersetup{pdftoolbar=true, pdfmenubar=true, pdffitwindow=false, pdfstartview={FitH}, pdftitle={\shorttitle}, pdfauthor={\shortauthors}, pdfsubject={\shorttitle}, pdfcreator={\shortauthors}, pdfproducer={\shortauthors}, pdfkeywords={}, pdfnewwindow=true, colorlinks=true, linkcolor=darkgray, citecolor=teal, filecolor=magenta, urlcolor=violet, linkbordercolor=red, citebordercolor=teal, urlbordercolor=violet, linktocpage=true}
    
\usetikzlibrary{arrows, decorations.markings, decorations, patterns,
positioning, decorations.pathreplacing, decorations.pathmorphing, external}
\tikzset{->-/.style={decoration={markings, mark=at position .5 with {\arrow{>}}},postaction={decorate}}}
\tikzset{-<-/.style={decoration={markings, mark=at position .5 with {\arrow{<}}},postaction={decorate}}}
\let\oldtocsubsection\tocsubsection
\renewcommand\tocsubsection[3]{\hspace{0.5cm}\oldtocsubsection{#1}{#2}{#3}}
\let\oldtocsubsubsection\tocsubsubsection
\renewcommand\tocsubsubsection[3]{\hspace{1cm}\oldtocsubsubsection{#1}{#2}{#3}}

\newcounter{res}[section]
\numberwithin{res}{section}
\newtheorem{thm}[res]{Theorem}
\newtheorem*{theo}{Theorem}

\newtheorem{lem}[res]{Lemma}
\newtheorem{lem-dfn}[res]{Lemma-Definition}

\newtheorem{prop}[res]{Proposition}
\newtheorem{cor}[res]{Corollary}

\theoremstyle{definition}
\newtheorem{notation}[res]{Notation}
\newtheorem{dfn}[res]{Definition}
\newtheorem{rmk}[res]{Remark}
\newtheorem{exa}[res]{Example}

 \let\oldmarginpar\marginpar
\newcommand\marginparL[1]{\oldmarginpar{\color{red}\fbox{\begin{minipage}{3cm}
  \footnotesize #1 \end{minipage}}}}

\newcommand\LHR[1]{{\bf\color{red}[LHR: #1]}}

\DeclareMathOperator{\shift}{\mathrm{sh}}
\DeclareMathOperator{\thickness}{\mathrm{th}}
\newcommand{\fone}[1][]{\ensuremath{\mathsf{\mathcal{F}}\ifstrempty{#1}{}{\left({#1}\right)}}}
\newcommand{\disk}{\ensuremath{\mathbb{D}}}
\newcommand{\myk}{\ensuremath{\mathtt{k}}}
\newcommand{\myN}{\ensuremath{\mathtt{N}}}
\newcommand{\scalars}{\ensuremath{R}}
\newcommand{\surface}{\ensuremath{\Sigma}}
\newcommand{\foam}{\ensuremath{F}}
\newcommand{\web}{\ensuremath{\Gamma}}
\newcommand{\ui}{\ensuremath{I}}
\newcommand{\algebra}[1][]{\ensuremath{\mathbb{A}\ifstrempty{#1}{}{_{#1}}}}
  \newcommand{\sg}[1][]{\ensuremath{\mathfrak{S}\ifstrempty{#1}{}{_{#1}}}}
    \newcommand{\comp}[1][]{\ensuremath{\underline{\ifstrempty{#1}{k}{#1}}}}

      \newcommand{\pos}[1][]{\ensuremath{\underline{\ifstrempty{#1}{p}{#1}}}}
        \newcommand{\sbim}[2]{\ensuremath{\leftindex_{#1}{B}_{#2}}}
        \newcommand{\vertex}[2]{\ensuremath{\leftindex_{#1}{V}_{#2}}}
        \newcommand{\iweb}[2]{\ensuremath{\leftindex_{#1}{I}_{#2}}}
        \newcommand{\tSSoergel}[2]{\ensuremath{\ifstrempty{#1}{}{#1\textrm{-}}\widetilde{\mathsf{SSoe}}\ifstrempty{#2}{}{_{#2}}}}
\newcommand{\tSSoe}{\tSSoergel{}{}}
\newcommand{\SSoergel}[2]{\ensuremath{\ifstrempty{#1}{}{#1\textrm{-}}{\mathsf{SSoe}}\ifstrempty{#2}{}{_{#2}}}}
    \newcommand{\SSoe}{\SSoergel{}{}}
    \newcommand{\blweb}[1][]{\ensuremath{\mathbf{BlWeb}\ifstrempty{#1}{}{_{#1}}}}
      
    \newcommand{\TblFoam}[1][]{\ensuremath{\mathsf{TopBlFoam}\ifstrempty{#1}{}{_{#1}}}}
    \newcommand{\blFoam}[1][]{\ensuremath{\mathsf{blFoam}\ifstrempty{#1}{}{_{#1}}}}
      \newcommand{\rep}[1][]{\ensuremath{\mathsf{Rep}\ifstrempty{#1}{}{\left({#1}\right)}}}
      \newcommand{\hTr}[1][]{\ensuremath{\mathrm{hTr}\ifstrempty{#1}{}{\left({#1}\right)}}}
        \newcommand{\tree}[1][\pos]{\ensuremath{\mathsf{Y}\ifstrempty{#1}{}{_{#1}}}}
  
\def\co{\colon\thinspace}

\newcommand{\NB}[1]{\ensuremath{\vcenter{\hbox{#1}}}}

\newcommand{\NN}{\ensuremath{\mathbb{N}}}
\newcommand{\ZZ}{\ensuremath{\mathbb{Z}}}

\newcommand{\QQ}{\ensuremath{\mathbb{Q}}}
\newcommand{\RR}{\ensuremath{\mathbb{R}}}
\newcommand{\DD}{\ensuremath{\mathbb{D}}}

\renewcommand{\SS}{\ensuremath{\mathbb{S}}}

  \newcommand{\id}{\mathrm{Id}}

\newcommand{\Hom}{{\mathrm{Hom}}}
\newcommand{\HOM}{{\mathrm{HOM}}}

\newcommand{\gll}{\ensuremath{\mathfrak{gl}}}

 \newcommand{\rk}{\mathrm{rk}}

\newcommand{\listk}[1]{\ensuremath{\underline{#1}}}

\renewcommand{\deg}[2][{}]{\ensuremath{\mathrm{deg}_{#1}(#2)}}

\newcommand{\BS}{\ensuremath{\mathcal{B}}}

\newcommand{\Qalggr}{\ensuremath{\QQ\textrm{-}\mathsf{Alg}_{\mathrm{gr}}}}

\newcommand{\Soergel}{\ensuremath{\mathsf{Soergel}}}

\newcommand{\seq}{\ensuremath{\mathrm{seq}}}

 \newcommand{\imagesfolder}{.}
\allowdisplaybreaks
\input{\imagesfolder/digon}

\input{\imagesfolder/assoc}

\input{\imagesfolder/square}

\begin{document}
\begin{abstract}
  The aim of this  short note is to establish a 2-equivalence between a certain 2-category of foams and that of singular Soergel bimodules of type A.

\end{abstract}
\maketitle
\setcounter{tocdepth}{2}
\tableofcontents

\section{Introduction} 
\label{sec:introduction}
Link homology theories have been developed and studied in the last three decades. They categorify quantum link invariants.  Different tools have been used to construct these homology theories. One of the most powerful but still somewhat mysterious theories is the HOMFLY-PT homology, also known as  the triply graded homology~\cite{KR}. It is defined using Hochschild homology of type $\mathsf{A}$ Soergel bimodules~\cite{KhHoch}.

HOMFLY-PT homology theory categorifies the two-variable HOMFLY-PT polynomial~\cite{HOMFLY, PT}, which can be specialized to a one-variable $\gll_\myN$-polynomial link invariant for any non-negative integer $\myN$. The latter is the Reshetikhin--Turaev quantum invariant for the quantum deformation  of the Lie algebra $\gll_\myN$ or $\mathsf{sl}_\myN$ and its fundamental representation \cite{MR939474,RT1}. Bigraded $\gll_\myN$ Khovanov--Rozansky link homology~\cite{KR0} is a categorification of the one-variable $\gll_\myN$-link invariant. This homology theory  was originally defined using matrix factorizations, but it has since been redefined in terms of foams~\cite{MR2443231,MR2491657},  which are suitable cobordisms between planar graphs. 
This construction used the Kapustin-Li evaluation formula for surfaces~\cite{MR2039036}, extended to foams in~\cite{MR2322554}.

The relationship between the HOMFLY-PT polynomial and the one-variable $\gll_\myN$-link invariant is categorified via the Rasmussen spectral sequence \cite{somediff}.

The Soergel category admits a \emph{two-dimensional} diagrammatical description that was discovered in type $\mathsf{A}$ in~\cite{EliasKhovanov}, see also~\cite{MR3502025} for the extension to the singular Soergel 2-category. In that description morphisms between tensor products of generating bimodules (Bott-Samelson bimodules) in the Soergel category are given by linear combinations of suitable diagrams modulo defining relations. Mackaay and Vaz~\cite{MR2671770} found a 3-dimensional lifting of that description, where Bott-Samelson tensor products of bimodules are represented by \emph{braid-like webs}  in the plane, while morphisms between them are given by linear combinations of 3-dimensional structures. These structures are \emph{braid-like foams} in $\RR^3$ with specific boundary conditions. More precisely, Mackaay and Vaz construct a collection of functors from the type $\mathsf{A}$ Soergel category to a suitable foam category (of braid-like foams) and show that the union of these functors is injective on morphisms. One can think of the foam description as a three-dimensional presentation of the Soergel category, while the diagrammatic description in~\cite{EliasKhovanov} is two-dimensional, and can be informally described as a \emph{holographic encoding} of the foam description. 

Another approach to foams was developed via  categorified Howe duality in~\cite{1212.6076,queffelec2014mathfrak,QRAnnular,RW, MR3906545} and many other papers. 
Here one restricts to \emph{ladder-like} or braid-like webs and foam cobordisms between them. Nondegeneracy of the resulting 3D structures follows via the comparison to the KLR algebras and full categorified quantum groups.  In this approach connection of foams to the Soergel category was established by Wedrich~\cite{2016arXiv160202769W}, who defined a 2-functor from the 2-category of ladder-like webs and foams to the 2-category of singular Soergel bimodules. This idea was underlying many of the developments afterwards, see \cite{RW2, qi2021categorificationcoloredjonespolynomial}.

The third approach to developing foam theory is via the Robert--Wagner foam evaluation formula~\cite{RW1,RW2}.
The present note relies on this approach. The aim of this note is to establish a 2-equivalence between the 2-category of singular Soergel bimodules and the 2-category of braid-like foams (called \emph{bl-foams} here). This solves the conjecture in \cite[Section 4.3]{RW2}, see  Theorem~\ref{thm:main} for a precise statement, and enhances the relation between foams and Soergel bimodules discovered and developed  in~\cite{MR2671770,2016arXiv160202769W}.

The main new result in the present paper concerns 2-morphisms in those two 2-categories and can be rephrased as follows:

 \begin{theo}
     Any morphism between singular Soergel bimodules can be uniquely realized as a linear combination of foams modulo $\infty$-equivalence.
\end{theo}

To finish, let us mention that the strengh of the foam perspective on singular Soergel bimodules is that it can drastically simplify algebraic computations, by further interpreting elements of singular Soergel bimodules themselves as foams, see Section \ref{sec:repr-2-funct} for this representation functor. Computations with foams motivated definitions and constructions in~\cite{qi2021categorificationcoloredjonespolynomial}.  In particular, this gives a way to easily obtain complicated explicit formulas for bimodule maps, and we give a sample example in Figure~\ref{fig:sample}. We believe this perspective can be pushed further, and we plan to do so in future work. 

We do not pass to the Karoubi completion in our construction, and what we call the singular Soergel 2-category may also be called the singular Bott--Samelson bimodule 2-category.

\begin{figure}[ht]
  \centering
  \NB{\tikz[]{\input{\imagesfolder/hhf_sample0}}}\quad
  $\leftrightsquigarrow$\quad
  \NB{\tikz[]{\input{\imagesfolder/hhf_sample}}} \hspace{4cm}\NB{\tikz[]{\input{\imagesfolder/hhf_sample3}}} \\[0.5cm]
  \NB{\tikz[]{\input{\imagesfolder/hhf_sample5}}}\quad$=$ \quad
  \NB{\tikz[]{\input{\imagesfolder/hhf_sample4}}}\quad   $\leftrightsquigarrow$\quad
  \NB{\tikz[]{\input{\imagesfolder/hhf_sample00}}}
  
  \caption{On the top row, the left part illustrates the correspondence between an element of a Soergel bimodule and its foamy counterpart: a tree-like bl-foam of elementtype, on the right, an example of a bl-foam of morphism type.  On the second row, the leftmost picture represents the bl-foam of element-type which is  the image of the (foamy representant) of the element of the Soergel bimodule of the top row by the foamy morphism of the top row. The equality is up to $\infty$-equivalence with $R$ being the symmetric polynomial {$\displaystyle{ \sum_{\substack{I\sqcup J = \{1, \dots, a+b\} \\ \#I = a, \#J =b}} \frac{P(x_I)Q(x_J)}{\prod_{\substack{i \in I \\ j \in J}}(x_i - x_j)}}$ (with $P$ (resp.{} $Q$) on a facet of thickness $a$ (resp.{} $b$))}.}
  \label{fig:sample}
\end{figure}
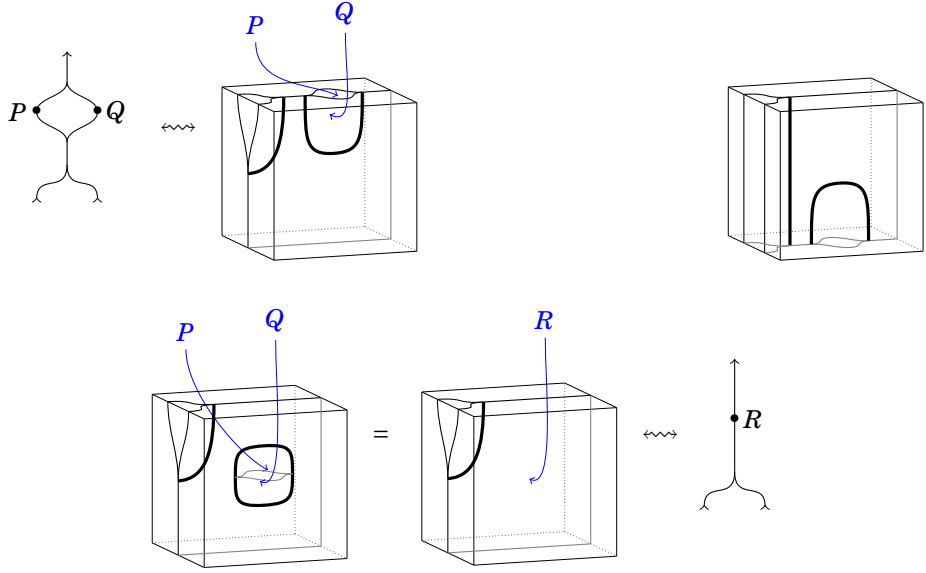

Besides this introduction, the paper is divided into three Sections. Section~\ref{sec:foams} presents the 2-category of bl-foams (\emph{braid-like foams}), Section~\ref{sec:Soergel}  presents that of singular Soergel bimodules (in type $\mathsf{A}$). Finally, Section~\ref{sec:2E} establishes the equivalence of 2-categories.

\subsection*{Acknowledgments}
M.K. was partially supported by NSF grant DMS-2204033 and Simons
Collaboration Award 994328 ``New Structures in Low-Dimensional Topology''.
L.-H.~R. was not supported by any project related funding. E.W. was partially supported by the ANR project Caget, ANR-25-CE40-4162. 

\subsection*{Conventions}

For symmetry reasons we grade spaces and homs between them by $\frac12\ZZ$ instead of $\ZZ$.

If $M=\bigoplus_{i \in \frac12 \ZZ} M_i$ is an object in a graded additive category and $s$ a half-integer, $q^{s}M$ denotes the graded object $M$ shifted up by $s$, in the sense that $(q^{s}M)_i= M_{i-s}$. More generally, if $P=\sum_{s}a_s q^{s} \in \ZZ_{\geq0}[q^{\pm\frac12}]$, define:
\[P(q)M := \bigoplus_s (q^{s}M)^{\oplus a_s}.\]
The balanced version of the quantum binomial is used throughout this paper, so that $[2]= q^{-1}+q$.

\section{Webs and foams} \label{sec:foams}

\subsection{(Com)positions}

Throughout the paper, $\myk$ denotes a non-negative integer. Recall that a \emph{composition of $\myk$} is a finite sequence of positive integers which sum up to $\myk$. Except otherwise stated, compositions are compositions of $\myk$. If $\comp=(k_1,\dots, k_\ell)$ is a composition, the integer $\shift(\comp)=\sum_{1\leq i<j\leq \ell} k_ik_j = \frac{1}{2} \sum_{i=1}^\ell k_i(\myk - k_i)$ is the \emph{shift} of $\comp$. The composition which has $\myk$ ones is denoted $1^\myk$. It has length $\myk$ and shift $\frac{\myk(\myk-1)}{2}$. The composition $(k)$, which has length $1$, is denoted $\comp[\myk]$ and has shift $0$.

A composition $\comp'$ is a \emph{decomposition} of a composition $\comp=(k_1, \dots, k_\ell)$ if $\comp'$ can be obtained by concatenation of compositions of $k_1,\dots, k_\ell$ {preserving their order}. In particular, the length of $\comp'$ is greater than or equal to $\comp$ and $\shift(\comp')$ is greater than or equal to $\shift(\comp)$. The composition $1^{\myk}$ is a decomposition of any composition, and any composition is a decomposition of $\comp[\myk]$. If the length $\ell'$ of decomposition $\comp'$ of $\comp$ is equal to $\ell+1$, with $\ell$ being the length of $\comp$, the decomposition $\comp'$ is  \emph{elementary}. For instance, $(3,4,1,2,4)$ is an elementary decompostion of $(3,4,3,4)$.

A \emph{position} is a pair $(\comp, \pos)$ where $\comp=(k_1,\dots, k_\ell)$ is a composition of $\myk$ and $\pos=(p_1, \dots, p_\ell)$ is a strictly increasing sequence of points in  $(0,1)$ of length $\ell$. One should think of the point $p_i$ carrying the integer $k_i$,  and we will therefore often denote the position $(\comp, \pos)$ by $\pos$. With this notation we define $\pi(\pos) := \comp$ to be ``composition part'' of the position $\pos$. The \emph{shift} of a position is that of its composition part $\pi(\pos)$, see above.

\subsection{Webs}
\subsubsection{Generalities}
A \emph{partially collared surface} is a pair $(\surface, \iota_\surface)$ where $\surface$ is a surface with a smooth embedding $\iota_\surface: \partial_c \surface \times [0;1[ \to \surface$, mapping $\partial_c \surface \times \{0\}$ identically to $\partial_c \surface$, where $\partial_c \surface$ is a codimension 0 closed sub-manifold of $\partial \surface$. 

Let $\surface$ be a partially collared surface. A \emph{web} in $\surface$ is an 
oriented uni-trivalent graph $\web$ smoothly embedded in $\surface$ with a thickness function $\thickness\co E(\web) \to \NN_{>0}$ satisfying a flow condition and a boundary condition:
\begin{itemize}
\item Flow condition: each trivalent vertex should follow one of two following local models:
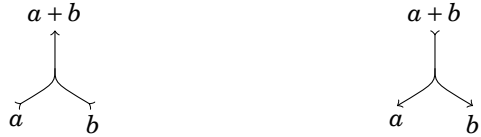
\[
\NB{\tikz[]{\input{\imagesfolder/sym_MOYvertex}}}
\]
Note that this is a condition on the function $\thickness$ and on the way $\web$ embeds  in $\surface$.
\item Boundary condition: the set of univalent vertices (denoted by $\partial \web$) is in $\partial_c \surface$ (i.e.{} the web $\web$ is properly embedded in $\surface$)  and for an $\epsilon >0$, $\web \cap \iota_\surface\big(\partial \surface \times[0;\epsilon)\big)= \iota_\surface\big((\partial\web)\times [0;\epsilon) \big).$ This conditions means that, in a neighborhood of its boundary, $\web$ is embedded in a way that is compatible with the collar of the surface. This enables gluing webs without troubles.  Note that the thickness function of $\web$ endowes each point of $\partial \web$ with a thickness. 
\end{itemize}

Suppose that $\surface$ is endowed with a Riemannian metric $g$. Then the condition imposed on the local models of trivalent vertices of a web $\web$ is so that the orientation of $\web$ defines a unit tangent vector $u_x$ at each point $x$ of $\web \subseteq \surface$. Suppose furthermore that $\surface$ is endowed with a vector field $V$.
A web $\web$  on $(\surface, g, V)$ is \emph{directed} if for all  $x\in \web$, $g(u_x, V_x)>0$.

Let $\ui$ denotes the standard unit interval $[0,1]$. In what follows we are interested in directed webs in $\disk:=\ui\times \ui\subseteq \RR^2$ with partial collar $\partial_c\disk = \ui\times\{0,1\}$, endowed with the standard Riemannian metric coming from the standard scalar product in $\RR^2$ and $V= \partial_y$. We call such webs  \emph{braid-like webs} or \emph{bl-webs}, for short. The boundary of a bl-web $\web$ consists of a position $\pos[p^0]$ in $\ui\times \{0\}$ and a position $\pos[p^1]$ in $\ui\times \{1\}$. One says that $\web$ is a bl-web from $\pos[p^0]$ to $\pos[p^1]$.  The directedness implies that these two positions are that of the same integer $\myk$.

\begin{figure}[ht]
  \centering
  \NB{\tikz[]{\input{\imagesfolder/hhf_bl-web-example}}}
  \caption{Example of a bl-web with $\myk= 11$. It is a morphism from $(4,5,2)$ to $(2, 4,2,3)$. }
  \label{fig:bl-web-exa}
\end{figure}

A bl-web such that all its trivalent vertices have different heights\footnote{By \emph{height} we mean second coordinate in $\ui^2$.} is said to be \emph{well-presented}.

For a later use, it convenient to introduce shifts for bl-webs: if $v$ is a merge vertex of a bl-web merging edges of thicknesses $a$ and $b$ into an edge of thickness $a+b$ or splitting an edge of thickness $a+b$ into edges of thickness $a$ and $b$, define the \emph{shift} of $v$ to be the element of $\frac12\ZZ$ denoted by $\shift(v)$ and equal to $\frac{ab}2$. The \emph{shift} of a bl-web $\web$ is denoted by $\shift(\web)$ and is equal to the sum of shifts of its vertices.

\subsubsection{The $1$-category of bl-webs}

One can consider the category $\blweb[\myk]$ of bl-webs of level $\myk$ where:
\begin{itemize}
\item Objects are positions of $\myk$.
\item Morphisms from position $\pos[p^0]$ to position $\pos[p^1]$ are bl-webs from $\pos[p^0]$ to $\pos[p^1]$, considered up to ambient isotopy (among bl-webs) rel boundary. 
\item Composition is given by stacking and rescaling. It is worth noticing that the (partially) collared condition ensures that the composition of two bl-webs is smooth.
\end{itemize}

\begin{rmk}
  \begin{enumerate}
  \item If a bl-web $\web: \pos[p^0] \to \pos[p^1]$ has no vertex, then $\pi(\pos[p^0])$ and $\pi(\pos[p^1])$ are equal and $\web$ is an isomorphism in $\blweb[\myk]$ (this isomorphism is actually unique).
  \item Conversely, if $\pos[p^0]$ and $\pos[p^1]$ are positions such that $\pi(\pos[p^0]) = \pi(\pos[p^1])$, then they are isomorphic in $\blweb[\myk]$ via a unique isomorphism. This isomorphism is denoted by $\iweb{\pos[p^0]}{\pos[p^1]}$.
  \item If a bl-web $\web: \pos[p^0] \to \pos[p^1]$ has a unique split (resp.{} merge) vertex, then  $\pi(\pos[p^1])$ (resp.{} $\pi(\pos[p^0])$) is an elementary decomposition of $\pi(\pos[p^0])$ (resp.{} $\pi(\pos[p^1])$). 
  \item If  $\web, \web': \pos[p^0] \to \pos[p^1]$ are two bl-webs with a unique vertex, then $\web = \web'$ in $\blweb[\myk]$, we denote this morphism by $\vertex{\pos[p^0]}{\pos[p^1]}$.
  \item Every bl-web is isotopic to a web for which trivalent vertices have different coordinates. Therefore all morphisms of $\blweb[\myk]$ are compositions of $\vertex{\bullet}{\bullet}$'s and $\iweb{\bullet}{\bullet}$'s. A bl-web which appears as such a composition is well-presented.\item Let $\web: \pos[p^0] \to \pos[p^1]$ be a bl-web. Consider $\web^{\dagger}$, the mirror image of $\web$ along $I \times \{\frac12\}$ with all orientations reversed. It is a bl-web from $\pos[p^1]$ to $\pos[p^0]$. This process endows $\blweb[\myk]$ with a dagger structure, that is, a contravariant involutive endo-functor $\dagger$ which is the identity on objects. \item The category $\blweb:= \bigsqcup_{\myk}\blweb[\myk]$ is equipped with a monoidal structure $\sqcup$ which comes from horizontal (disjoint) stacking and rescaling.
  \item   The dagger and monoidal structures of $\blweb$ are compatible in the sense that:
    \((\web_1\sqcup \web_2)^\dagger = \web_1^\dagger \sqcup \web_2^\dagger.\)
  \end{enumerate}
\end{rmk}

\subsection{Foams}
\label{sec:sub-foams}

\subsubsection{Generalities}
\label{sec:subsub-foams}

\begin{dfn}\label{dfn:foam}
  A \emph{foam} $\foam$ is a collection
  of \emph{facets} $\foam^{(2)}=(\Sigma_i)_{i\in I}$, that is a
  finite collection of oriented connected surfaces with boundary,
  together with the following data:
  \begin{itemize}\item A \emph{thickness} function
    $\thickness \co (\Sigma_i)_{i\in I} \to  \NN_{>0}$. \item A set of \emph{decorations}, that is, for each facet $f \in \foam^{(2)}$, a
    symmetric polynomial $P_f$ in $\thickness(f)$ variables with rational
    coefficients.  A facet $f$ is called \emph{trivially decorated} if
    $P_f=1$.
  \item A ``gluing recipe'' of the facets along their boundaries such
    that upon gluing there are the three possible local models: the neighborhood of a point is homeomorphic to  {\it
    (i)}  a surface,
    {\it (ii)} a tripod times an interval or {\it (iii)} the cone of
    the 1-skeleton of the tetrahedron, as depicted below.
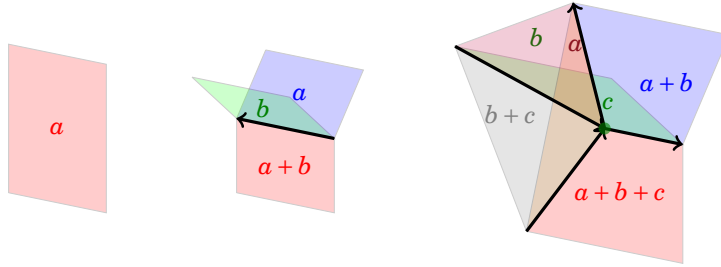
\[
  \begin{tikzpicture}
        \input{\imagesfolder/cef_3localmodels}

      \end{tikzpicture}
      \label{fig:FBSP} \]
The letter appearing on a facet is the thickness of this facet.  Thus
    a foam consists of \emph{facets}, \emph{bindings} or \emph{bindings} (which are connected compact oriented $1$-manifolds) and \emph{singular vertices}. Each
    binding carries:
    \begin{itemize}
    \item An orientation which agrees with the orientations of the
      facets with thickness $a$ and $b$ and disagrees with the
      orientation of the facet with thickness $a+b$. Such a binding has
      \emph{type $(a, b, a+b)$}.
\item A cyclic ordering of the three facets around it. When a foam
      is embedded in $\RR^3$,  we require this cyclic ordering to agree
      with the left-hand rule\footnote{This agrees with Khovanov's
      convention used in~\cite{MR2100691}.} with respect to its
      orientation (the dotted circle in the middle indicates that the
      orientation of the binding points to the reader).
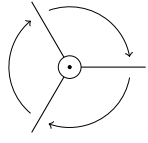
\[
        \begin{tikzpicture}[xscale=1]
          \input{\imagesfolder/sw_lhrule}
        \end{tikzpicture}
      \]
    \end{itemize}
    The cyclic orderings of the different bindings adjacent to a
    singular vertex should be compatible. This means that a
    neighborhood of the singular vertex is embeddable in $\RR^3$ in a
    way that respects the left-hand rule for the four binding adjacent
    to this singular vertex.
  \end{itemize}
  In particular, when forgetting about its thickness  decorations and
  orientations, a foam has a structure of a compact, finite 2-dimensional
  CW-complex.

  Denote by $\mathcal{S}$ the collection of circles which are
  boundaries of the facets of $F$. The gluing recipe consists of: 
\begin{itemize}
  \item A subset $\mathcal{S}'$ of $\mathcal{S}$ and a subdivision of
    each circle of $\mathcal{S'}$ into a finite number of closed
    intervals. This gives us a collection $\mathcal{I}$ of closed
    intervals.
  \item Partitions of
    $\mathcal{I} \cup (\mathcal{S} \setminus \mathcal{S'})$ into
    subsets of three elements. For every subset $(Y_1, Y_2, Y_3)$ of
    this partition, three diffeomorphisms $\phi_1 : Y_2 \to Y_3$,
    $\phi_2 : Y_3 \to Y_1$, $\phi_3 : Y_1 \to Y_2$ are fixed such that
    $\phi_3 \circ \phi_2 \circ \phi_1 = \mathrm{id}_{Y_2}$.
  \end{itemize}
\end{dfn}
Bindings of a foam are circles (elements of $\mathcal{S}\setminus \mathcal{S}'$) and intervals (included in the circles of $\mathcal{S}'$), and the above gluing recipe explains how to glue triples of facets along common binding intervals and circles.

The concept of foam extends naturally to the
concept of \emph{foam with boundary}. The boundary of a foam has the
structure of a web. We require that the facets and bindings are
locally orthogonal to the boundary in order to be able to glue them
together canonically. When we want to emphasize that a foam has no boundary, we say that it is \emph{closed}.

\begin{dfn}
Let $\foam$ be a closed foam and fix an integer $\myN$ greater than or equal to the maximum thickness of facets of $\foam$. One can form a surface $\Sigma_\myN(\foam)$ by replacing each facet of thickness $a$ by $a(\myN-a)$ copies of that facet. These parallel copies are glued together along seams and singular vertices to obtain $\Sigma_\myN(\foam)$. The \emph{$\myN$-degree} of $\foam$ is the integer $\deg[\myN]{\foam}$ defined by:
\begin{equation}
  \label{eq:1}
\deg[\myN]{\foam} :=-\chi(\Sigma_\myN(\foam)) + 2\sum_{f\in \foam^{(2)}} \deg{P_f}.   
\end{equation}
\end{dfn}
In the language of~\cite{RW1}, $a(\myN-a)$ is the number of bicolored surfaces for a given $\gll_\myN$-coloring which contain that facet, with $a$ choices for the pigment in the facet and $\myN-a$ choices for the color not in the facet.

When the foam $\foam$ is not closed the same formula can be used, but the degree is additionally shifted by a quantity depending only on the boundary of $\foam$ (see Definitions~\ref{dfn:morphisms-type} and \ref{dfn:element-type}).

\subsubsection{The $2$-category of bl-foams}

Consider the cube $C = \ui^3$ with the following
parameterization of its boundary:
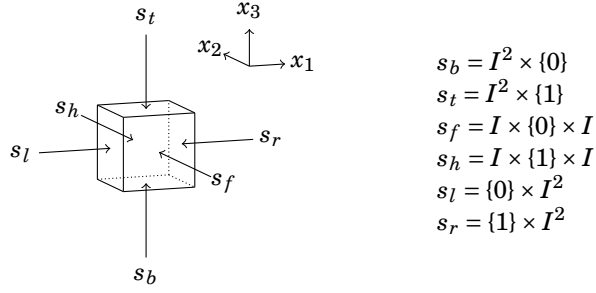
\begin{figure}[h]
  \centering
  \[
    \NB{\tikz[]{\input{\imagesfolder/sym_cube}}}
    \qquad \qquad
    \begin{array}{l}
      s_b = I^2\times \{0\} \\
      s_t = I^2\times \{1\} \\
      s_f = I\times \{0\} \times I\\
      s_h = I\times \{1\} \times I\\
      s_l = \{0\} \times I^2\\
      s_r = \{1\} \times I^2
    \end{array}
  \] 
  \caption{Parametrization and notations for the standard cube.}
\label{fig:cube}
\end{figure}
The symbols $s_\bullet$ denote the 6 squares of the boundary of $C$
and the letters $f, h, l, r, b, t$ stand for \textbf{f}ront,
\textbf{h}idden, \textbf{l}eft, \textbf{r}ight, \textbf{b}ottom and
\textbf{t}op.  The plane $P$ is parallel to the squares $s_f$ and $s_h$.
\begin{dfn}
  \label{dfn:foamincube}
  Let $F$ be a foam with boundary embedded in $C$. Suppose that the
  boundary of $F$ is contained in $s_l\cup s_r\cup s_b\cup s_t$, and
  that the MOY-graphs $F\cap s_l$, $F\cap s_r$, $F\cap s_b$ and
  $F\cap s_t$ are all braid-like. We say that $F$ is a \emph{bl-foam}
  if for every point $x$ of $F$, the normal line of the foam $F$ at
  $x$ is \emph{not} parallel to $P$.

  Let $\pos[p^0]$ and $\pos[p^1]$ two positions and $\web_0, \web_1 \colon \pos[p^0] \to \pos[p^1]$ two bl-webs. If a bl-foam $\foam$ is such that $\foam \cap s_b = -\web_0$,  $\foam \cap s_t = \web_1$, $\foam \cap s_l = -\pos[p^0]\times I $ and,  $\foam \cap s_r = \pos[p^1] \times I$, then $F\foam$ is \emph{of morphism type}   and we write: $\foam \colon \web_0\to \web_1$.
\end{dfn}

\begin{dfn}\label{dfn:morphisms-type}
  Let $\web_0, \web_1 \colon \web_0, \web_1 \colon \pos[p^0] \to \pos[p^1]$ are two bl-webs and $\foam\colon \web_0 \to \web_1$ is a foam of morphism type. Define the degree of $\foam$ to be:
  \begin{equation}
    \deg{\foam} := -\chi(\Sigma_\myk(\foam)) + 2\sum_f \deg{P_f}  + \frac12\sum_{i=1}^{\ell^0} k^0_i(\myk-k_i^0) + \frac12\sum_{j=1}^{\ell^1} k^1_j(\myk - k_j^1),
  \end{equation}
 where $\comp[k^0]= (k^0_1, \dots k^0_{\ell^0})$ (resp.{} $\comp[k^1]= (k^1_1, \dots k^1_{\ell^1})$) is  the composition of $\myk$ underlying $\pos[p^0]$ and $\pos[p^1]$.
\end{dfn}

With this definition we can form the 2-category $\TblFoam[\myk]$:
\begin{itemize}
\item Objects are positions of $\myk$,
\item $1$-morphisms are bl-webs,
\item $2$-morphisms are bl-foams of morphism type, considered up to isotopy rel boundary.
\end{itemize}

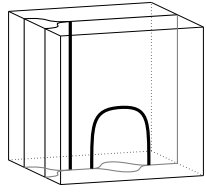
\begin{figure}
\centering
\NB{\tikz[]{\input{\imagesfolder/hhf_sample3}}}
\caption{Example of a foam of morphism type.}
\end{figure}

Compositions of $1$-morphisms and $2$-morphisms is given my gluing and resizing. For the composition of the $2$-morphisms, decorations are multiplicative: the decoration of a facet $f$ of $F \circ G$  in the composition is obtained as the product of the decorations of the facets of $F$ and $G$ which constitute $f$. The degree of bl-foams or morphism type behaves additively with respect to vertical and horizontal compositions.

\begin{rmk}
  The 2-category $\TblFoam := \displaystyle{\bigsqcup_{\myk \geq 0} \TblFoam[\myk]}$ has a monoidal structure given by taking the disjoint union and rescaling along the $x_2$-direction (see Figure~\ref{fig:cube} for the parametrization).

\end{rmk}

As suggested by its name, $\TblFoam[\myk]$ is a category of a topological flavor, similar to that of cobordisms. To introduce a more algebraic version of that category we need  the notion of $\infty$-equivalence.

\subsubsection{$\infty$-equivalence}
\label{sec:infty-equivalence}

In \cite{RW1}, for any integer $\myN$, the $\gll_\myN$-evaluation of closed foams was defined. We denote it by $\tau_\myN$. It associates with every closed foam $\foam$ an element $\tau_\myN(\foam) \in  \ZZ[X_1, \dots, X_\myN]^{\sg[\myN]}$. Two $\QQ$-linear combinations of foams $\sum_i \lambda_iF_i$ and $\sum_j \mu_jG_j$ with common boundary $\web$ are \emph{$\infty$-equivalent} if for any foam $H$ with boundary $-\web$ and any integer $\myN$,  the following identity holds:
\[
  \sum_i \lambda_i \tau_\myN\left(H\circ F_i\right) =
  \sum_j \mu_j \tau_\myN\left(H\circ G_j\right). 
\]
As its name suggests, $\infty$-equivalence is an equivalence relation, it is homogeneous, so that spaces of foams modulo $\infty$-equivalence are naturally graded.

We can now form the 2-category $\blFoam[\myk]$:
\begin{itemize}
\item Objects are compositions of $\myk$,
\item $1$-morphisms are direct sums of $\frac12\ZZ$-shifted bl-webs,
\item $2$-morphisms $2\Hom$ are matrices of formal $\QQ$-linear combinations of foams with the degree prescribed by the differences of grading shifts, with the linear combinations considered up to $\infty$-equivalence.
\end{itemize}

As usual when dealing with graded categories, if $\web_0, \web_1\colon \pos[p^0] \to \pos[p^1]$ are two bl-webs, define the graded $\QQ$-vector space by summing over 2-morphisms for all grading shifts: 
\[
  2\HOM_{\blFoam}(\web_0, \web_1) = \bigoplus_{i\in \frac12\ZZ}  2\Hom_{\blFoam}(\web_0,q^{i}\web_1).
\]

Similarly to $\TblFoam := \displaystyle{\bigsqcup_{\myk \geq 0} \TblFoam[\myk]}$, one can form $\blFoam:= \displaystyle{\bigsqcup_{\myk \geq 0} \blFoam[\myk]}$. It inherits a monoidal structure given by disjoint union and rescaling along the $x_2$-direction. 

We won't give much details on the foam evaluation formula here. However for future use we record a few local relations on foams which hold up to $\infty$-equivalence  and conclude this section with a result by Queffelec--Rose~\cite{QRAnnular} which will be crucial for proving  Theorem~\ref{thm:main}.

\begin{lem}[Dot migration, {\cite[Equation (11)]{RW1}}]
  \label{lem:foam-dot-migration}
If $P \in \scalars[x_1, \dots, x_{a+b}]^{\sg[a+b]}$ is rewritten as a $\sum_{i} Q_i^{(1)}\otimes Q_i^{(2)} \in \scalars[x_1, \dots, x_{a}]^{\sg[a]}\otimes \scalars[x_{a+1}, \dots, x_{a+b}]^{\sg[b]} $, then up to $\infty$-equivalence, the following local identity holds:

\begin{equation}
  \NB{\tikz[]{\input{\imagesfolder/hhf_dot-migration-foam}}}.
\end{equation}
\end{lem}

\begin{lem}[Bubble removal, {\cite{RW1}}, {\cite[p.17]{qi2021categorificationcoloredjonespolynomial}}]
  The following identity holds up to $\infty$-equivalence:
  \label{lem:bubble}
  \begin{equation}
     \NB{\tikz[]{\input{\imagesfolder/hhf_bubble-2}}}
   \end{equation}
   where the bubble consists of two facets: a disk of thickness $a$ (depicted on the top) and one disk of thickness $b$ (depicted on the bottom). The quantity $R$ is a symmetric polynomial in $a+b$ variables. 
\end{lem}

\begin{lem}[Diabolo removal, {\cite{RW1}, {\cite[(4-3)]{RW}}}]
  \label{lem:diabolo}
  The following identity holds up to $\infty$-equivalence:
  \begin{equation}
     \NB{\tikz[]{\input{\imagesfolder/hhf_2-sided-dish-2}}}
   \end{equation}
   where the foam on the LHS consists of three facets: two annuli of thicknesses $a$ (on the top) and $b$ (on the bottom), and a disk of thickness $a+b$ in the middle. These three facets have a common circular boundary. The foam on the RHS consists of two disjoint disks of thickness $a$ (on the top) and $b$ (on the bottom). The sum runs over all Young diagram $\lambda$ with at most $a$ rows and $b$ columns  and $\widehat{\lambda}$ is the transpose of the complementary of $\lambda$ in the rectangle $a\times b$, so that it has at most $b$ rows and $a$ columns. Finally, $s_\lambda$ is the Schur polynomial in $a$ variables associated with $\lambda$ and $s_{\widehat{\lambda}}$ is the Schur polynomial in $b$ variables associated with $\widehat{\lambda}$, and ${|\widehat{\lambda}|}$ is the number of boxes in ${\widehat{\lambda}}$. 
\end{lem}

\begin{prop}[{\cite[Theorem 3.30]{RW1}}]
  \label{prop:iso-web}
  We have the following isomorphisms between $1$-morphims in $\blFoam$:
  \begin{gather}
    \NB{\tikz[]{\input{\imagesfolder/hhf_assoc}}}, \qquad \qquad \qquad      \NB{\tikz[]{\input{\imagesfolder/hhf_co-assoc}}}, \\
    \NB{\tikz[]{\input{\imagesfolder/hhf_digon}}}, \\
    \NB{\tikz[]{\input{\imagesfolder/hhf_square-1}}} \qquad \text{if $b-a + d-c \geq 0$,} \\
    \NB{\tikz[]{\input{\imagesfolder/hhf_square-2}}} \qquad \text{if $a-b + d-c \geq 0$.}
\end{gather}
\end{prop}

In this paper we are interested in $\blFoam$, in the proof of the main theorem, where we need to deal with annular webs and foams. Categorically, this means we work in the horizontal trace of the $2$-category $\blFoam$. The objects of $\hTr[\blFoam]$ are 1-endomorphisms of $\blFoam$. Morphisms from $\web_0 \colon \pos[p^0] \to \pos[p^0]$ to  $\web_1 \colon \pos[p^1] \to \pos[p^1]$ are equivalence classes of pairs $(\web, \foam)$, where $\web \colon \pos[p^0] \to \pos[p^1]$ and $F \colon \web\circ \web_0  \to \web_1 \circ \web$ is a 2-morphism. The equivalence relation is generated by imposing that $(\web, (\id_{\web_1} \cdot \foam') \circ \foam)$ is equivalent to $(\web', \foam'\circ(\foam \cdot \id_{\web_0}))$ for $\web, \web' \colon \pos[p^0] \to \pos[p^1]$, $\foam \colon \web\circ \web_0  \to \web_1 \circ \web'$  and $\foam' \colon \web' \to \web$. Vertical composition of $2$-morphisms is denoted by $\circ$, and $\cdot$ denotes horizontal composition of $2$-moprhisms.  See \cite[Section~4.3]{queffelec2018annular} for more details.

More concretely, morphisms in $\hTr[\blFoam]$ are linear combinations of $\infty$-equivalence classes of annular bl-foams. By bl-foam here, we mean that generic cross-sections are bl-webs.

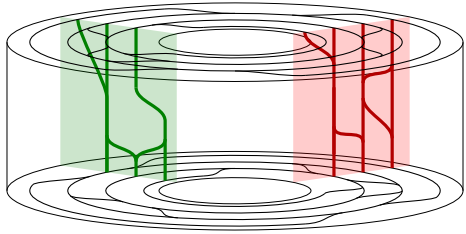
\begin{figure}[ht]
  \centering
  \NB{\tikz[]{\input{\imagesfolder/hhf_annular-foam-example}}}
  \caption{Schematic of the constraints of an annular bl-foam}
  \label{fig:exa-annular-foam}
\end{figure}

In $\hTr[\blFoam]$, the objects $\id_{\pos}$ will play a special role. If two positions $\pos[p^0]$ and $\pos[p^1]$ have the same underlying decomposition then the objects $\id_{\pos[p^0]}$ and $\id_{\pos[p^1]}$ are isomorphic in $\hTr[\blFoam]$. For each composition $\comp$ of $\myk$ we chose one position $\pos$ realizing it and denote $\id_{\pos}$ (seen in $\hTr[\blFoam]$) by $\SS_{\comp}$.

\begin{prop}[Queffelec--Rose algorithm~{\cite[Lemma~5.2]{QRAnnular}}]
  \label{prop:QR}
  Let $\web \colon \pos \to \pos$ be an object of $\hTr[\blFoam]$, then there exist families of polynomials $(P_{\comp}(\web))_{\comp}$ and $(Q_{\comp}(\web))_{\comp}$ in  $\ZZ_{\geq 0}[q, q^{-1}]$ indexed by compositions of $\myk$, such that:
  \[
    \web \oplus \left(\bigoplus_{\comp} Q_{\comp}(\web) \SS_{\comp}\right) \simeq \bigoplus_{\comp} P_{\comp}(\web) \SS_{\comp}.
  \]
  Moreover, this decomposition is algorithmic in the sense that it can be achieved by using iteratively the isomorphisms presented in Proposition~\ref{prop:iso-web}.
\end{prop}

This algorithm works for both setups:  (1) foams constructed via categorified quantum groups, as in~\cite{QRAnnular}, and (2) foams constructed via the Robert--Wagner evaluation~\cite{RW1} and its $\infty$-equivalence modification for braid-like and annular foams in~\cite{RW2} and the present paper.

\section{Soergel bimodules}
\label{sec:Soergel}

\subsection{Fixing convention}
\label{sec:soergel-convention}
The aim of this section is to present the 2-category of singular
Soergel bimodules in type $\mathsf{A}$ (mostly to fix conventions and notations).
In what follows, $\scalars$ is a commutative ring, and the polynomial
algebra $\scalars[x_1, \dots, x_\myk]$ is denoted by $\algebra$. It is
graded by imposing that the variables $x_i$ are homogeneous of degree
2.  It is naturally acted upon by $\sg[\myk]$, the symmetric group on
$\myk$ elements.

For a composition $\comp=(k_1, \dots, k_\ell)$, $\sg[\comp]$ denotes
the parabolic subgroup of $\sg[\myk]$ (isomorphic to
$\prod_{i=1}^\ell \sg[k_i]$) naturally associated with that
composition. The algebra $\algebra^{\sg[\comp]}$ is denoted by
$\algebra[\comp]$, so that in particular, $\algebra[1^\myk]= \algebra$
and $\algebra[{\comp[\myk]}]$ is the algebra of symmetric polynomials
in $x_1, \dots, x_{\myk}$.

If $\comp'$ is a decomposition of $\comp$, then $\sg[\comp']$ is a subgroup of $\sg[\comp]$, and, therefore, $\algebra[\comp]$ is a sub-algebra of $\algebra[\comp']$. Consequently, the (shifted) algebra $q^{\frac{\shift(\comp)-\shift(\comp')}{2}} \algebra[\comp']$ can be viewed as a $(\algebra[\comp], \algebra[\comp'])$-bimodule (resp.{} as a $(\algebra[\comp'], \algebra[\comp])$-bimodule). This bimodule is denoted $\sbim{\comp}{\comp'}$ (resp.{} $\sbim{\comp'}{\comp}$). When writing $\sbim{\comp[k_1]}{\comp[k_2]}$, it is implicitly assumed that $\comp[k_1]$ is a decomposition of $\comp[k_2]$, or, conversely, that $\comp[k_1]$ is a decomposition of $\comp[k_2]$. Such bimodules and their tensor products are called \emph{singular Bott--Samelson} bimodules.

Let $\scalars\text{-}\mathsf{Alg}_{\frac{1}{2}\mathbb{Z}}$ be the 2-category of $\frac{1}{2}\mathbb{Z}$-graded $\scalars$-algebras, where:
\begin{itemize}
\item Objects are $\frac{1}{2}\mathbb{Z}$-graded $\scalars$-algebras;
\item 1-morphisms are $\frac{1}{2}\mathbb{Z}$-graded bimodules;
\item 2-morphisms are bimodule maps.
\end{itemize}
Composition of $1$-morphisms is given by taking tensor product over the ``middle'' algebras. 

The 2-category $\tSSoergel{\scalars}{\myk}$ is a sub-2category of $\scalars\text{-}\mathsf{Alg}_{\frac{1}{2}\mathbb{Z}}$:
\begin{itemize}
\item Objects are algebras $\algebra[\comp]$ where $\comp$ runs over compositions of $\myk$;
\item 1-morphisms are direct sums of shifted versions of compositions\footnote{The polysemous  nature of \emph{composition} is slightly annoying; we hope this will not lead to any confusion. Here, of course, we mean composition as 1-morphisms.} of singular Bott--Samelson bimodules.
\item 2-morphisms are bimodule maps. 
\end{itemize}
When no confusion is possible, we write $\tSSoe$ instead of $\tSSoergel{\scalars}{\myk}$. The literature typically focuses on the $2$-category $\SSoe$, the Karoubi completion of $\tSSoe$, in other words, $\SSoe$ is defined similarly expect that, for $1$-morphisms, one may also take direct summands.

\subsection{Graphical representation}
\label{sec:graph-repr}

\newcommand{\cbim}{\ensuremath{\mathbf{bim}}}

Before investigating the 2-category $\tSSoe$, it is worth stating a few basic facts about  bimodules. For this it is convenient temporarily to discard part of the 2-category structure and consider the $1$-category $\cbim$: \begin{itemize}
\item Objects are graded commutative $\scalars$-algebras,
\item Morphisms are graded bimodules up to isomorphisms,
\item Composition is given by tensor product over algebras.
\end{itemize}
The tensor product over $\scalars$ induces a monoidal structure on $\cbim$. 

If $A$ and $B$ are commutative algebras, then any $(A,B)$-bimodule can be considered as a $(B,A)$-bimodule. This simple fact induces a dagger structure on $\cbim$. 

(Singular) Soergel bimodules have been used for defining the (colored) triply graded homology. In this context, it is  useful to represent them by diagrams. Formally this amounts to saying that one uses a $1$-functor $\fone$ from the category $\blweb$ to $\cbim$ defined as follows. On objects, 
$\fone(\pos[p]) =  \algebra[{\pi(\pos[p])}]$.
On morphisms, one defines $\fone(\iweb{\pos[p]}{\pos[p]})= \id_{\algebra[{\pi(\pos[p])}]}$ (that is, $\algebra[{\pi(\pos[p])}]$ is seen as a $(\algebra[{\pi(\pos[p])}],\algebra[{\pi(\pos[p])}])$-bimodule) and $\fone(\vertex{\pos[p]}{{\pos[p']}})= \sbim{\pos[p]}{\pos[p']}$. Since $\pos[p]$ is trivially a decomposition over itself, one also has $\fone(\iweb{\pos}{\pos}) = \sbim{\pos}{\pos}$.

This indeed defines a functor $\fone$: the only thing to check is that two isotopic  well-presented bl-web are mapped onto isomorphic bimodules. The isotopies that change the order (relative height) of the trivalent vertices induce canonical isomorphisms of bimodules. 
 The functor $\fone$ is monoidal and respects the dagger structure. 

It is worth noting that one associates actual bimodules (as opposed to isomorphism classes) to well-presented bl-webs. The functor $\fone$ allows us to use a diagrammatic language to speak about singular Soergel bimodules.

If $\web$ is a bl-web, the elements of $\fone(\web)$ can be represented as linear combination of decorations on the edges of the web $\web$ subject to dot-migration relations at vertices. A \emph{decoration} of an edge of thickness $a$ is a symmetric polynomial in $a$ variables. The dot-migration relations can be expressed as follows:
if $P \in \scalars[x_1, \dots, x_{a+b}]^{\sg[a+b]}$ is rewritten as a $\sum_{i} Q_i^{(1)}\otimes Q_i^{(2)} \in \scalars[x_1, \dots, x_{a}]^{\sg[a]}\otimes \scalars[x_{a+1}, \dots, x_{a+b}]^{\sg[b]} $, then we have:

\[
  \NB{\tikz[]{\input{\imagesfolder/hhf_dot-migration-web}}}  \qquad\text{and}\qquad
    \NB{\tikz[]{\input{\imagesfolder/hhf_dot-migration-web-2}}}.
\]

For legibility, the functor $\fone$ will be omitted as soon as the context allows and diagrams of bl-webs will be considered as bimodules.  This is, for example, the case in the next proposition.

Let $(D_e)_{e \in E(\web)}$ be a collection of homogeneous decoration. It represents an element $D$ of  $\fone[\web]$. Taking in account grading shifts in $\fone[\web]$, the degree of $D$ is equal to $2 \sum \deg{D_e} - \shift(\web)$.

\begin{prop}[Folklore, see for instance discussion in {\cite[Appendix A]{hogancamp2021skeinrelationsingularsoergel}} or {\cite[Section 5.2]{queffelec2018annular}}] \label{prop:soergel-skein}
  Bimodules associated (by $\fone$) to well-presented bl-webs satisfy the following identities:
  \begin{gather}
    \NB{\tikz[]{\input{\imagesfolder/hhf_assoc}}}, \qquad \qquad \qquad      \NB{\tikz[]{\input{\imagesfolder/hhf_co-assoc}}}, \\
    \NB{\tikz[]{\input{\imagesfolder/hhf_digon}}}, \\
    \NB{\tikz[]{\input{\imagesfolder/hhf_square-1}}} \qquad \text{if $b-a + d-c \geq 0$,} \\
    \NB{\tikz[]{\input{\imagesfolder/hhf_square-2}}} \qquad \text{if $a-b + d-c \geq 0$.}
\end{gather}
\end{prop}

\begin{rmk}
  \begin{enumerate}
  \item The monoidality of the functor $\fone$ implies that these relations holds ``locally'': they may be surrounded by vertical strands on both sides.
\item The existence of the functor $\rep$ defined in Section~\ref{sec:repr-2-funct} and relations on webs also implies Proposition~\ref{prop:soergel-skein}.
  \end{enumerate}
\end{rmk}

\subsection{About 2-morphisms}
\label{sec:about-2-morphisms}

Let $\comp$ be a composition. Recall that $\sbim{\comp}{\comp}$ is the algebra $\algebra[\comp]$ considered as a $(\algebra[\comp], \algebra[\comp])$-bimodule (with no shift). It is the identity morphism of the object $\algebra[\comp]$. 
The space of endomorphisms of $\sbim{\comp}{\comp}$ is isomorphic to the algebra $\algebra[\comp]$. It is a free graded $\scalars$-module, and, if $\comp = (k_1, \dots, k_\ell)$, its graded rank over $\scalars$ is equal to \begin{align}\prod_{i=1}^\ell \frac{1}{1-q^{2k_i}}.\end{align}
It is also a free graded $\algebra[{\comp[\myk]}]$-module and its graded rank over $\algebra[{\comp[\myk]}]$ is equal to \begin{align}\prod_{i=1}^\ell  q^{\shift(\comp)} \
\begin{bmatrix}
  \myk \\ k_1 \,\, \cdots \,\,\, k_\ell
\end{bmatrix}.
\end{align}

\newcommand{\op}{{\ensuremath{\mathrm{op}}}}

Recall that, if $\comp'$ is a decomposition of $\comp$ then $\algebra[\comp]$ is a sub-algebra of $\algebra[{\comp'}]$, and that  $\sbim{\comp}{\comp'}$ (resp.{} $\sbim{\comp}{\comp'}$) is, up to a grading shift, $\algebra[{\comp'}]$ seen as a $(\algebra[\comp], \algebra[{\comp'}])$-bimodule (resp.{} as a
$(\algebra[{\comp'}], \algebra[{\comp}]$-bimodule). Hence the functors
\begin{align}
  \sbim{\comp}{\comp'}\otimes_{\algebra[{\comp'}]} \cdot\thinspace: \algebra[{\comp'}]\text{-}\mathsf{Mod} \to \algebra[{\comp}]\textrm{-}\mathsf{Mod} \qquad \text{and} \qquad
  \sbim{\comp'}{\comp}\otimes_{\algebra[{\comp}]} \cdot \thinspace: \algebra[{\comp}]\text{-}\mathsf{Mod} \to \algebra[{\comp'}]\textrm{-}\mathsf{Mod}.
\end{align}
are the induction and restriction functors. They form a bi-adjunction, being both left and right adjoint (since $\algebra[{\comp'}]$ is graded symmetric Frobenius over $\algebra[{\comp}]$). Taking in account grading shifts, one has, for any graded $(\algebra[\comp], \algebra[{\comp[\ell]}])$-bimodule $M$ and $(\algebra[{\comp'}], \algebra[{\comp[\ell]}])$-bimodule $N$:   
\begin{align}
  &\HOM_{\algebra[{\comp'}] \otimes \algebra[{\comp[\ell]}]^\op}(\sbim{\comp'}{\comp}\otimes_{\algebra[{\comp}]}M,N) \simeq q^{\shift(\comp') - \shift(\comp)}
  \HOM_{\algebra[\comp] \otimes \algebra[{\comp[\ell]}]^\op}(M,\sbim{\comp'}{\comp}\otimes_{\algebra[{\comp}]}N), \qquad \text{and}\\
  &\HOM_{\algebra[{\comp}] \otimes \algebra[{\comp[\ell]}]^\op }(\sbim{\comp}{\comp'}\otimes_{\algebra[{\comp'}]}N,M) \simeq q^{\shift(\comp) - \shift(\comp')}
  \HOM_{\algebra[\comp] \otimes \algebra[{\comp[\ell]}]^\op}(N,\sbim{\comp'}{\comp}\otimes_{\algebra[{\comp}]}M).
\end{align}
Similarly, for any graded $(\algebra[{\comp[\ell]}], \algebra[\comp])$-bimodule $M$ and $(\algebra[{\comp[\ell]}], \algebra[{\comp'}],)$-bimodule $N$:   
\begin{align}
  &\HOM_{\algebra[{\comp[\ell]}] \otimes \algebra[{\comp'}]^\op}(M \otimes_{\algebra[{\comp}]} \sbim{\comp}{\comp'},N) \simeq q^{\shift(\comp') - \shift(\comp)}
  \HOM_{\algebra[{\comp[\ell]}] \otimes \algebra[\comp]^\op}(M,N \otimes_{\algebra[{\comp'}]} \sbim{\comp'}{\comp}), \qquad \text{and}\\
  &\HOM_{\algebra[{\comp[\ell]}] \otimes \algebra[{\comp}]^\op}(N\otimes_{\algebra[{\comp'}]} \sbim{\comp'}{\comp},M) \simeq q^{\shift(\comp') - \shift(\comp)}
  \HOM_{\algebra[{\comp[\ell]}] \otimes \algebra[\comp]^\op}(N,M \otimes_{\algebra[{\comp}]} \sbim{\comp}{\comp'}).
\end{align}
for any graded right $\algebra[\comp]$-module $M$ and any graded right $\algebra[{\comp'}]$-module $N$.   

\begin{cor}
  If $M$ and $N$ are two compositions of Bott--Samelson bimodules, $M$ being an $(\algebra[\comp], \algebra[{\comp'}])$-bimodule and $N$  an $(\algebra[{\comp'}], \algebra[{\comp}])$-bimodule, then:
  \begin{align}
    \HOM_{\algebra[{\comp'}]^\mathrm{e}}(\algebra[{\comp'}], M \otimes_{\algebra[{\comp}]} N) \simeq
        q^{2\shift(\comp') - 2\shift(\comp)}\HOM_{\algebra[{\comp}]^\mathrm{e}}(\algebra[{\comp}], N \otimes_{\algebra[{\comp'}]} M). \end{align}
\end{cor}

\section{2-equivalence}
\label{sec:2E}

\subsection{A representation 2-functor}
\label{sec:repr-2-funct}

The aim of this part is to describe a $2$-functor $\rep$ from $\blFoam[\myk]$ to $\tSSoergel{\QQ}{\myk}$. 
We think of this 2-functor as a representation of a topological 2-category of braid-like foams into the algebraic 2-category of Soergel bimodules. 
 This 2-functor  maps bl-webs to homs between Soergel bimodules, and it also allows to describe elements of Soergel bimodules in terms of foams.

From now on, we suppose that $\scalars= \QQ$, so that $\tSSoe = \tSSoergel{\QQ}{\myk}$.  
 If $\pos$ is a position and $\comp=\pi(\pos)$ its composition part, we set $\rep[\pos] = \algebra[\comp]$.
In order to describe $\rep$ on $1$-morphisms, we need to introduce standard trees and foams of element type.

Recall that $\comp[\myk]$ is the composition of $\myk$ of length $1$. Let it also denote the position $(\pos[\myk], \frac12)$, that is the position which consists of  one point of $I$ (its middle of point) with weight $\myk$. If $\pos$ is a position, denote by $\comp=(k_1, \dots, k_\ell)$ its composition part $\pi(\pos)$. Let us fix a standard bl-web $\tree[\pos]: \pos[\myk] \to \pos$ shaped as a tree and described by the following diagram:
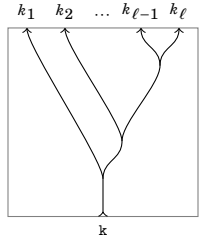
\[
  \NB{\tikz[]{\input{\imagesfolder/hhf_std-tree2}}}
\]
The precise choice of the tree is not important; we fix such a tree for each position.

\begin{dfn}
  \label{dfn:element-type}
  Let $\pos[p^0]$ and $\pos[p^1]$ be two positions and $\web\colon \pos[p^0] \to \pos[p^1]$ be a bl-web. A bl-foam $\foam$ is of \emph{element type for $\web$} if
  $\foam \cap s_b = -\pos[\myk] \times I$,  $\foam \cap s_t = \web$, $\foam \cap s_l = -\tree[{\pos[p^0]}] $ and  $\foam \cap s_r = \tree[{\pos[p^1]}]$ (see Figure~\ref{fig:bl-foam-example}).
  The \emph{degree} of a foam $\foam$ of element type is given by:
  \begin{equation}
    \label{eq:2}
    \deg{\foam}:= -\chi(\Sigma_\myk(\foam)) + 2\sum_f \deg{P_f} .
  \end{equation}
  A bl-foam $F$ of element type is \emph{tree-like} if for  any $t \in I$, $F\cap \left(\{t\} \times I^2\right)$ is a tree, and if all non-trivial decorations are carried by facets intersecting $s_t$.
\end{dfn}
\begin{figure}[ht]
  \centering
  \NB{\tikz[]{\input{\imagesfolder/hhf_bl-foam-example}}}
  \caption{Example of the boundary condition of a bl-foam of element type.}
  \label{fig:bl-foam-example}
\end{figure}
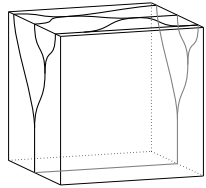

In tree-like foams, non-trivial decorations are only placed on facets that  bound the top bl-web. The tree-like condition also ensures that there is a canonical correspondence between such facets and the edges of the top bl-web. Hence, for two tree-like foams for $\web$, we can compare their decorations and, in particular, make sense of the hypothesis of Lemma~\ref{lem:tree-like-2}.

\begin{lem}[{\cite[Lemmas 3.30]{RW2}}]
  \label{lem:tree-like-1}
  Every element type foam is $\infty$-equivalent to a linear combination of tree-like foams.
\end{lem}

\begin{lem}[{\cite[Lemmas 3.32]{RW2}}]
    \label{lem:tree-like-2}
If two tree-like foams for $\web$ have the same decorations, they are $\infty$-equivalent.
\end{lem}

For $\web$ a bl-web, denote by $\rep[\web]$ the space of $\QQ$-linear combinations of foams of element type for $\web$, regarded up to $\infty$-equivalence. We can  think of $\rep[\web]$ as  foams cobordisms (subject to boundary and bl-foam conditions) from an interval of thickness $\myk$ to $\web$, modulo the $\infty$-equivalence relation.

Let $\web$ be a bl-web and $M := \fone[\web]$ be the corresponding Soergel bimodule. Since dot-migration relations on foams follow from the $\infty$-equivalence (see Lemma~\ref{lem:foam-dot-migration}), Lemma \ref{lem:tree-like-2} implies that we can define a map from $M$ to $\rep[\web]$, by mapping a decoration on $\web$ to the  corresponding tree-like foam. Lemma~\ref{lem:tree-like-1} indicates that this map is surjective. Comparing graded dimensions of $M$ and $\rep[\web]$, one obtains that this map is an isomorphism.

\begin{lem}[{\cite[Section 4.3]{RW2}}]
  \label{lem:foamy-soergel-1}
  For any bl-web $\web$, $\rep[\web]$ is isomorphic to $\fone[\web]$ as a graded $\QQ$-vector space. 
\end{lem}
Let $\web\colon \pos[p^0] \to \pos[p^1]$ be a bl-web.  Multiplication by decorations on the facets touching $\tree[{\pos[p^0]}]$ and $\tree[{\pos[p^1]}]$ endows $\rep[\web]$ with an $(\rep[{\pos[p^0]}], \rep[{\pos[p^1]}])$-bimodule structure. This structure is compatible with the isomorphism of Lemma~\ref{lem:foamy-soergel-1}. Moreover it is also compatible with the composition of bl-webs and that of bimodules.
Hence, for any bl-web $\web$, $\rep[\web]$ is isomorphic to $\fone[\web]$ as a Soergel bimodule.

Let $\web_0, \web_1 \colon \pos[p^0] \to \pos[p^1]$ be two bl-webs,  $F\colon \web_0 \to \web_1$ a bl-foam of morphism type and $E$ a bl-foam of element type for $\web_0$. Form $F(E)$, a bl-foam of element type for $\web_1$, as depicted below:
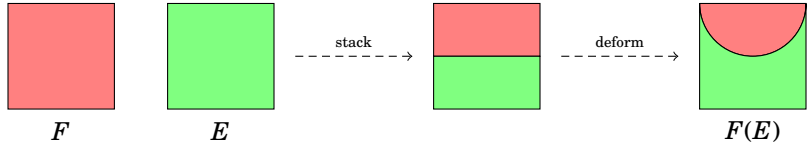
\begin{align*}
\NB{\tikz[]{\input{\imagesfolder/hhf_stackingFE-2}}}
\end{align*}

Let us describe this procedure: the bl-foams $E$ and $F$ are stacked onto one another and scaled vertically. This produces a bl-foams which is almost of morphism type except that the trees on the sides have longer leaf segments due having attached $F$, which consists of parallel vertical segments on the sides. One then deforms the foam to shrink the leaf lines  and stretch the rest of the two trees back to their original embeddings.

This procedure is compatible with $\infty$-equivalence, so that $F$ induces a $\ZZ$-linear map $\rep[F] \colon \rep[\web_0] \to \rep[\web_1]$. Moreover, it commutes with action by $\rep[{\pos[p^0]}]$ and $\rep[{\pos[p^1]}]$, so that $\rep[F]$ is a bimodule map.

This whole discussion can be summarized as follows. 

\begin{prop}[{\cite[Proposition 4.15]{RW2}}]
  \label{prop:foamy-soergel-1}
 The map $\rep$ is a $2$-functor from $\blFoam[\myk]$ to $\tSSoe$.
\end{prop}

This paper establishes the following result. 

\begin{thm}
  \label{thm:main}
  The $2$-functor $\rep$ is an equivalence of $2$-categories. 
\end{thm}

\begin{proof}
  The functor $\rep$ is  split essentially surjective: for each composition $\comp$ of $\myk$, one has $\algebra[\comp] \simeq \rep[\pos]$ for $\pos$ the evenly spaced position carrying $\comp$. Hence it remains to show that $\rep$ is fully faithful. In this context, this means that, for any positions $\pos[p^0]$ and $\pos[p^1]$, the functor induced by $\rep$ on the category $\Hom(\pos[p^0], \pos[p^1])$ is essentially surjective and fully faithfull. Once more, essential surjectivity is easy:  by definition of $\tSSoe$, objects are directed sums of shifted version of bimodules coming from bl-webs. The only  remaining thing to show is full faithfulness of this functor. That is, that it induces bijections between spaces of $2$-morphisms.
  
Instead of looking at 2-morphisms of degree $0$ between shifted versions of 1-morphisms, we actually work with graded version of 2HOM-spaces between unshifted $1$-morphisms. 
  Hence the theorem follows from the following lemma.
   \begin{lem}\label{lem:main}
    Let $\web_0, \web_1\colon \pos[p^0] \to \pos[p^1]$ be two bl-webs of index $\myk$. The 2-functor $\rep$ induces an isomorphism of graded vector spaces 
    \[2\HOM_{\blFoam[\myk]}(\web_0, \web_1)\qquad \cong\qquad 2\HOM_{\tSSoe}(\rep[\web_0], \rep[\web_1]). \qedhere \] 
  \end{lem}
\end{proof}

\subsection{Proof of Lemma~\ref{lem:main}}
\label{sec:proof-theor-refthm:m}

The proof of Lemma~\ref{lem:main} is done in two steps: first reduce to a
simpler setting, then  prove the assertion in this simpler
setting.

The key arguments for reduction are the compatibility of $\rep$ with the isomorphisms described in Section~\ref{sec:about-2-morphisms} and the Queffelec--Rose algorithm \cite{queffelec2014mathfrak}.

\begin{lem}
  \label{lem:reduction}
  It is enough to prove  Lemma~\ref{lem:main} in the case $\pos[p^0] =  \pos[p^1]$ and $\web_0=\web_1 = \pos[p^0] \times I$. In other words, we can only consider the case where both $\web_0$ and $\web_1$ are identity $1$-morphisms.
\end{lem}

\begin{proof}
  First, we can use isomorphisms of 1-morphisms in the $2$-category $\blFoam$ to reduce the problem. We begin by proving that in addition to using the isomorphisms listed in Proposition~\ref{prop:iso-web} one can move webs along the cube.

  Recall that if $\web \colon \pos[p^0] \to \pos[p^1]$, $\web^\dagger$ is the bl-web from $\pos[p^1]$ to $\pos[p^0]$ obtained by mirror symmetry and orientation reversing.
  The isomorphisms of Section~\ref{sec:about-2-morphisms} coming from bi-adjoint pair  the induction/restriction can conveniently be phrased using $\rep$ and $\dagger$:
  \begin{align*}
&    2\HOM_{\tSSoe}\left(\rep[\web_1\circ\web_0], \rep[\web_3 \circ \web_2] \right) \,\, \simeq \\
&\qquad\qquad\qquad
q^{\shift(\pos[p^0_0]) + \shift(\pos[p^0_2]) -\shift(\pos[p^0_1]) -\shift(\pos[p^1_1]) }2\HOM_{\tSSoe}\left(\rep[\web_0\circ\web_2^\dagger], \rep[\web_1^\dagger \circ \web_3] \right).
  \end{align*}
where $\web_0\colon \pos[p_0^0] \to \pos[_1^0]$, $\Gamma_1\colon \pos[p_1^0] \to \pos[p_2^0]$, $\Gamma_2\colon \pos[p_0^0] \to \pos[p_1^1],$ and  $\Gamma_3\colon \pos[p_1^1] \to \pos[p_0^2]$.
On the foamy side, with the same notation, one also has isomorphisms:
    \[2\HOM_{\blFoam[\myk]}(\web_1\circ\web_0, {\web_3 \circ \web_2} ) \simeq
q^{\shift(\pos[p^0_0]) + \shift(\pos[p^0_2]) -\shift(\pos[p^0_1]) -\shift(\pos[p^1_1]) } 2\HOM_{\blFoam[\myk]}({\web_0\circ\web_2^\dagger}, {\web_1^\dagger \circ \web_3} ).
  \]
  which comes from the following deformation of bl-foams:
  \[
    \NB
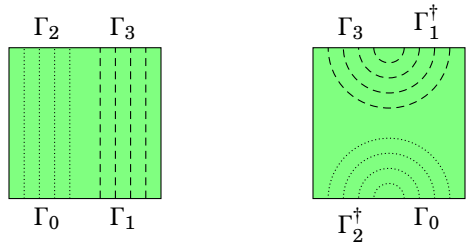
{\tikz[]{\input{\imagesfolder/hhf_defoamation}}}
  \]
  From the foamy description of Soergel bimodules given in Lemma~\ref{lem:foamy-soergel-1}, one deduces that for any compatible bl-webs $\web_0, \web_1, \web_2, \web_3$, the square
  \[
    \tikz[xscale=5, yscale =2]{
      \node (A) at (0,2) {$2\HOM_{\blFoam[\myk]}(\web_0\circ\web_1, {\web_2 \circ \web_3} )$};
      \node (B) at (1,1) {$2\HOM_{\tSSoe}\left(\rep[\web_0\circ\web_1], \rep[\web_2 \circ \web_3] \right)$};
      \node (D) at (0,0) {$q^{\shift(\web_3) - \shift(\web_0)} 2\HOM_{\blFoam[\myk]}({\web_1\circ\web_3^\dagger}, {\web_0^\dagger \circ \web_2} )$};
      \node (C) at (1,-1) {$q^{\shift(\web_3) - \shift(\web_0)} 2\HOM_{\tSSoe}\left(\rep[\web_1\circ\web_3^\dagger], \rep[\web_0^\dagger \circ \web_2] \right)$};
    \draw[->] (A) -- (B) node[midway, above] {\rep};
    \draw[->] (D) -- (C) node[midway, above] {\rep};
  \draw[->] (A) -- (D) node[midway, above, sloped] {deformation};
  \draw[->] (B) -- (C) node[midway, above, sloped] {biadjointness};}
  \]
commutes. Since the vertical arrows are isomorphisms, the top (slanted) horizontal one is an isomorphism if and only if the bottom horizontal one is.

Hence, the full faithfullness of $\rep$ is an annular question and can be seen through the Queffelec--Rose algorithm. Then, thanks to Proposition~\ref{prop:QR}, the only case to inspect is that when the annular webs are concentric circles. This is precisely what the lemma claims. 
\end{proof}

The main ingredient to conclude in the simpler case is a theorem by Beliakova--Habiro--Lauda--Webster \cite{MR3653092} about horizontal trace of 2-categories categorifying quantum groups which, when read in terms of foams, explains how certain foams can be simplified up to $\infty$-equivalence. 
To apply it, we consider braid-like foams in the cylinder that bound collections of concentric circles on each side, as follows.

Suppose that $F$ is a foam in a thickened annulus $A\times I$ which can be obtained as the trace of a bl-web $F'\colon \web \to \web$: it bounds collections  of concentric and  coherently oriented circles on top and bottom of the annulus. Moreover, it satisfies a condition of direction: generic cross sections of $F$ are bl-webs and, in particular, point upwards. Foams satisfying the direction condition are called \emph{vinyl} foams.

\begin{prop}[{\cite[Theorem 3.2 (second part)]{queffelec2018annular}}]
  \label{prop:trace-like}
  Suppose that foam $F$ is the trace of a bl-web $F'\colon \web \to \web$. Then $F$ is $\infty$-equivalent to a $\QQ$-linear combinations of decorated version of foams with shape $\web \times \SS^1$.
\end{prop}

Another way to put it, is to say that, if the boundary of a vinyl foam has an $\SS^1$-symmetry, then, up to $\infty$-equivalence, it is $\SS^1$-symmetric.

\begin{proof}[Proof of Lemma~\ref{lem:main} in the simpler case.]
  Let us suppose that $\web_0 = \web_1 = \id_{\pos}$ and denote by $\comp=(k_1, \dots, k_\ell)$ the composition $\pi(\pos)$. We aim to prove that $\rep$ induces an isomorphism from $2\HOM_{\blFoam[\myk]}(\web_0, \web_0)$ to $2\HOM_{\tSSoe}(\rep[\web_0], \rep[\web_0])$.
  The former consists of $\QQ$-linear combination of bl-foams bounding $\SS_{\comp}$ (concentric circles with thicknesses given by $\comp$) up to $\infty$-equi\-va\-lence. The latter is the endomorphism ring of $\algebra[\comp]$ considered as a bimodule over itself and is isomorphic to $\algebra[\comp]$.

  The map induced by $\rep$ is surjective: indeed, if $P=P_1\otimes \dots \otimes P_\ell$ is a factorizable element of $\algebra[\comp]$, the bimodule endomorphism corresponding to $P$ is equal to $\rep[\DD_P]$, where $\DD_P$ is the collection of disks of thicknesses $k_1, \dots, k_\ell$ and decorated by $P_1, \dots, P_\ell$. 
  
  The non-trivial part is injectivity. It is enough to prove that any bl-foam bounding $\SS_{\comp}$ is $\infty$-equivalent to a $\QQ$-linear combination of decorated disks. Let $\foam$ be such a bl-foam. Recall that $\foam$ is in the standard cube $C$ (see Figure~\ref{fig:cube}). It is a bl-foam  transverse to the $x_2$-direction . Let $U$ be a thickened segment in the $x_2$ direction, so that $U\cap \foam$ is a disjoint union of (small) disks (of various thicknesses and placement which are parametrized by, say, $\pos'$). Make $U$ an even smaller thickening, to make sure that $\foam\setminus U$ is a foam denoted by $\mathring{\foam}$.

  The foam $\mathring{\foam}$ is embedded in $C\setminus U$, which is a thickened annulus, and satisfies the direction properties which makes it fall into the hypothesis of Proposition~\ref{prop:trace-like}. Hence, it is equivalent to a $\QQ$-linear combination of a decorated version of foams $\web \times \SS^1$ for various bl-webs $\web\colon \pos' \to \pos$.

  Remember that $\foam$ is recovered from $\mathring{\foam}$ by capping it with disks along $\SS_{\pos}$. Hence it is enough to prove that if $G$ is a decorated, capped-off (on one side) foam $\web \times \SS^1$ for $\web\colon \pos' \to \pos$ a bl-web, then $G$ is $\infty$-equivalent to a $\QQ$-linear combination of disjoint union of decorated disks. We argue by induction on the number of vertices of $\web$. If $\web$ has no vertex, then $G$ is already a disjoint union of decorated disks. Otherwise, consider a vertex $v$ of $\web$ as close as possible to $\pos'$. It is either a merge or a split.

  Suppose that $v$ is a merge vertex. It means that $G$ contains a bubble. This bubble can be removed up to $\infty$-equivalence.  This is Lemma~\ref{lem:bubble},   illustrated below. 
  \[
\NB
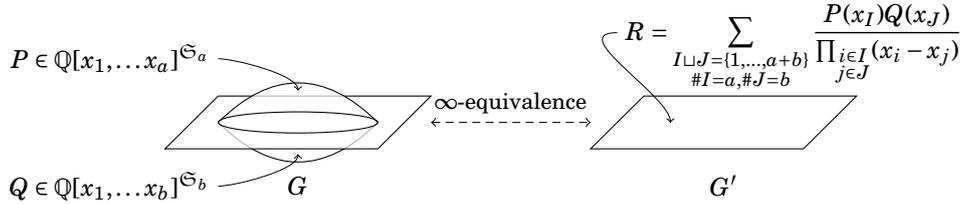
{\tikz[]{\input{\imagesfolder/hhf_bubble-1}}}
\]
The resulting foam  $G'$ can be obtained by capping off a decorated version of $\web'\times \SS^1$ with $\web'$ with one less vertex that $\web$, so that we can apply the induction.

  Suppose now that $v$ is a split vertex. It means that $G$ contains a diabolo (a fringed disk or a two-sided dish). Thanks to the dot migration, we can assume that decorations of the disks is trivial. This diabolo can be removed, up to $\infty$-equivalence. This is Lemma~\ref{lem:diabolo}, illustrated below.
    \[
\NB
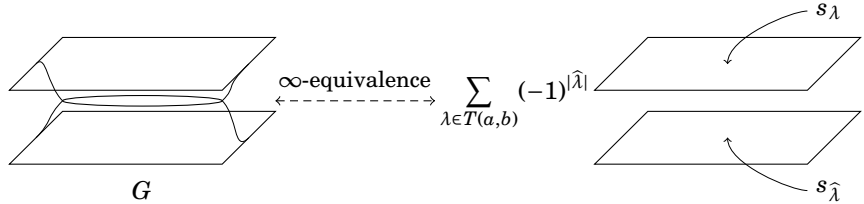
{\tikz[]{\input{\imagesfolder/hhf_2-sided-dish-1}}}
\]
The resulting linear combination consists of foams $G'$ that can be obtained by capping off decorated versions of $\web'\times \SS^1$, where $\web'$ has one less vertex that $\web$. We can now conclude the proof by induction.\end{proof}

\bibliographystyle{alphaurl}
\bibliography{biblio}

\end{document}

%% file: digon.tex
\newcommand{\digona}{\ensuremath{\vcenter{\hbox{\tikz[scale=0.4]{
\coordinate (B) at (0,0);
\coordinate (V1) at (0,0.5);
\coordinate (V2) at (0,2.5);
\coordinate (T) at (0,3);
\draw[white] (0, -0.5) -- (0, 3.5);
\draw[->] (B) -- (V1) node[pos=0, below] {\tiny{$a+b$}};
\draw[->] (V2) -- (T) node[pos=1, above] {\tiny{$a+b$}};
\draw[->-] (V1) .. controls +(+0.5, +0.5) and +(+0.5, -0.5).. (V2) node[midway, right] {\tiny{$b$}};
\draw[->-] (V1) .. controls +(-0.5, +0.5) and +(-0.5, -0.5).. (V2) node[midway, left] {\tiny{$a$}};
}}}}}

\newcommand{\verta}{\ensuremath{\vcenter{\hbox{\tikz[scale=0.4]{
\coordinate (B) at (0,0);
\coordinate (T) at (0,3);
\draw[white] (0, -0.5) -- (0, 3.5);
\draw[->] (B) -- (T) node[pos=0, below] {\tiny{$a+b$}};
}}}}}

%% file: assoc.tex
\newcommand{\stgamma}{\ensuremath{\vcenter{\hbox{\tikz[scale=0.3]{
\coordinate (B) at (0,0);
\coordinate (V1) at (0,1);
\coordinate (V2) at (1,2);
\coordinate (T1) at (-2,3);
\coordinate (T2) at (0,3);
\coordinate (T3) at (2,3);
\draw[>-] (B) -- (V1) node [at start, below] {\tiny{$a+b+c$}};
\draw[->] (V1) -- (T1) node [at end, above] {\tiny{$a$}};
\draw[->-] (V1)  -- (V2) node[midway, right] {\tiny{$b+c$}};
\draw[->] (V2) -- (T2) node[at end, above] {\tiny{$b$}};
\draw[->] (V2) -- (T3) node[at end, above] {\tiny{$c$}};
}}}}}

\newcommand{\stgammaprime}{\ensuremath{\vcenter{\hbox{\tikz[scale=0.3]{
\coordinate (B) at (0,0);
\coordinate (V1) at (0,1);
\coordinate (V2) at (-1,2);
\coordinate (T1) at (-2,3);
\coordinate (T2) at (0,3);
\coordinate (T3) at (2,3);
\draw[>-] (B) -- (V1) node [at start, below] {\tiny{$a+b+c$}};
\draw[->] (V1) -- (T3) node [at end, above] {\tiny{$c$}};
\draw[->-] (V1)  -- (V2) node[midway, left] {\tiny{$a+b$}};
\draw[->] (V2) -- (T1) node[at end, above] {\tiny{$a$}};
\draw[->] (V2) -- (T2) node[at end, above] {\tiny{$b$}};
}}}}}

\newcommand{\stgammar}{\ensuremath{\vcenter{\hbox{\tikz[scale=0.3, yscale = -1]{
\coordinate (B) at (0,0);
\coordinate (V1) at (0,1);
\coordinate (V2) at (1,2);
\coordinate (T1) at (-2,3);
\coordinate (T2) at (0,3);
\coordinate (T3) at (2,3);
\draw[<-] (B) -- (V1) node [at start, above] {\tiny{$a+b+c$}};
\draw[-<] (V1) -- (T1) node [at end, below] {\tiny{$a$}};
\draw[-<] (V1)  -- (V2) node[midway, right] {\tiny{$b+c$}};
\draw[-<] (V2) -- (T2) node[at end, below] {\tiny{$b$}};
\draw[-<] (V2) -- (T3) node[at end, below] {\tiny{$c$}};
}}}}}

\newcommand{\stgammaprimer}{\ensuremath{\vcenter{\hbox{\tikz[scale=0.3, yscale = -1]{
\coordinate (B) at (0,0);
\coordinate (V1) at (0,1);
\coordinate (V2) at (-1,2);
\coordinate (T1) at (-2,3);
\coordinate (T2) at (0,3);
\coordinate (T3) at (2,3);
\draw[<-] (B) -- (V1) node [at start, above] {\tiny{$a+b+c$}};
\draw[-<] (V1) -- (T3) node [at end, below] {\tiny{$c$}};
\draw[-<] (V1)  -- (V2) node[midway, left] {\tiny{$a+b$}};
\draw[-<] (V2) -- (T1) node[at end, below] {\tiny{$a$}};
\draw[-<] (V2) -- (T2) node[at end, below] {\tiny{$b$}};
}}}}}

%% file: square.tex
\newcommand{\easysquarec}{\ensuremath{\vcenter{\hbox{\tikz[xscale=0.65, yscale=0.55]{
\coordinate (B1) at (-1,0);
\coordinate (B2) at (1,0);
\coordinate (C1) at (-1,1.1);
\coordinate (D1) at (-1,1.9);
\coordinate (C2) at (1,0.9);
\coordinate (D2) at (1,2.1);
\coordinate (T1) at (-1,3);
\coordinate (T2) at (1,3);
\draw[>-] (B1) -- (C1) node[at start, below] {\tiny{$a$}};
\draw[->-] (C1) -- (D1) node[midway, left   ] {\tiny{$a+1$}};
\draw[->] (D1) -- (T1) node[at end , above ] {\tiny{$a$}};
\draw[>-] (B2) -- (C2) node[at start, below] {\tiny{$b$}};
\draw[->-] (C2) -- (D2) node[midway, right] {\tiny{$b-1$}};
\draw[->] (D2) -- (T2) node[at end, above] {\tiny{$b$}};
\draw[->-] (D1) -- (D2) node[midway, above] {\tiny{$1$}};
\draw[->-] (C2) -- (C1) node[midway, below] {\tiny{$1$}};
}}}}}

\newcommand{\easyIIupup}{\ensuremath{\vcenter{\hbox{\tikz[xscale=0.65, yscale=0.55]{
\coordinate (B1) at (-1,0);
\coordinate (B2) at (1,0);
\coordinate (C1) at (-1,1.1);
\coordinate (D1) at (-1,1.9);
\coordinate (C2) at (1,0.9);
\coordinate (D2) at (1,2.1);
\coordinate (T1) at (-1,3);
\coordinate (T2) at (1,3);
\draw[->] (B1) -- (T1) node[at start, below] {\tiny{$a$}};
\draw[->] (B2) -- (T2) node[at start, below] {\tiny{$b$}};
}}}}}

\newcommand{\easysquared}{\ensuremath{\vcenter{\hbox{\tikz[yscale=0.55, xscale=0.65]{
\coordinate (B1) at (-1,0);
\coordinate (B2) at (1,0);
\coordinate (C1) at (-1,0.9);
\coordinate (D1) at (-1,2.1);
\coordinate (C2) at (1,1.1);
\coordinate (D2) at (1,1.9);
\coordinate (T1) at (-1,3);
\coordinate (T2) at (1,3);
\draw[>-] (B1) -- (C1) node[at start, below] {\tiny{$a$}};
\draw[->-] (C1) -- (D1) node[midway, left   ] {\tiny{$a-1$}};
\draw[->] (D1) -- (T1) node[at end , above ] {\tiny{$a$}};
\draw[>-] (B2) -- (C2) node[at start, below] {\tiny{$b$}};
\draw[->-] (C2) -- (D2) node[midway, right] {\tiny{$b+1$}};
\draw[->] (D2) -- (T2) node[at end, above] {\tiny{$b$}};
\draw[->-] (D2) -- (D1) node[midway, above] {\tiny{$1$}};
\draw[->-] (C1) -- (C2) node[midway, below] {\tiny{$1$}};
}}}}}



%% file: hhf_sample0.tex
\begin{scope}[scale=2]
  \tdplotsetmaincoords{0}{0}
  \begin{scope}[tdplot_main_coords]
      \draw[>->] (0.7, 0, 1)
      .. controls +(0, +0.2, 0) and  +(0, -0.2, 0) .. (0.5, 0.25, 1)
      .. controls +(0, +0.1, 0) and  +(0, -0.1, 0) .. (0.5, 0.4, 1)
      .. controls +(0, +0.1, 0) and  +(0, -0.1, 0) .. (0.7, 0.6, 1) coordinate (q) node {$\bullet$}
      .. controls +(0, +0.1, 0) and  +(0, -0.1, 0) .. (0.5, 0.8, 1) -- (0.5,1,1);
      \draw (0.5, 0.4, 1)
      .. controls +(0, +0.1, 0) and  +(0, -0.1, 0) .. (0.3, 0.6, 1) coordinate (p) node {$\bullet$}
      .. controls +(0, +0.1, 0) and  +(0, -0.1, 0) .. (0.5, 0.8, 1);
      \draw[>-] (0.3, 0, 1)
      .. controls +(0, +0.2, 0) and  +(0, -0.2, 0) .. (0.5, 0.25, 1);
      \node[left] at (p) {$P$};
      \node[right] at (q) {$Q$};

  \end{scope}
\end{scope}


%% file: hhf_sample.tex
\begin{scope}[scale=2]
  \tdplotsetmaincoords{80}{70}
  \begin{scope}[tdplot_main_coords]
      \draw[<-, very thin, blue] (0.3, 0.7, 0.95) .. controls +(-0.5, 0, 0) and +(0,0, -0.3) .. (-1, 0.6, 1.1) node[above] {$P$};
    \coordinate (OOO) at (0, 0, 0);
    \coordinate (IOO) at (1, 0, 0);
    \coordinate (OIO) at (0, 1, 0);
    \coordinate (IIO) at (1, 1, 0);
    \coordinate (OOI) at (0, 0, 1);
    \coordinate (IOI) at (1, 0, 1);
    \coordinate (OII) at (0, 1, 1);
    \coordinate (III) at (1, 1, 1);
    \coordinate (sf) at (+3.5,  0.5,  0.5);
    \coordinate (sh) at (-2.5,  0.5,  0.5);
    \coordinate (sr) at ( 0.5, +2.25,  0.5);
    \coordinate (sl) at ( 0.5, -1.25,  0.5);
    \coordinate (st) at ( 0.5,  0.5, +2.25);
    \coordinate (sb) at ( 0.5,  0.5, -1.25);
    \begin{scope}[very thin]
      \draw (OOO) -- (OOI);
      \draw [densely dotted] (OOO) -- (OIO);
      \draw (OOO) -- (IOO);
      \draw (III) -- (OII);
      \draw (III) -- (IOI);
      \draw (III) -- (IIO);
      \draw (OOI) -- (OII);
      \draw (OOI) -- (IOI);
      \draw [densely dotted] (OIO) -- (OII);
      \draw [densely dotted] (OIO) -- (IIO);
      \draw (IOO) -- (IOI);
      \draw (IOO) -- (IIO);
    \end{scope}
    \begin{scope}
      \draw (0.5, 0, 0) .. controls +(0, 0, 0.2) and +(0, 0, -0.2) .. (0.5, 0, 0.5)
      .. controls +(0, 0, 0.2) and +(0, 0, -0.2) .. (0.7, 0, 1);
      \draw (0.5, 0, 0.5)
      .. controls +(0, 0, 0.2) and +(0, 0, -0.2) .. (0.3, 0, 1);
      \draw (0.7, 0, 1)
      .. controls +(0, +0.2, 0) and  +(0, -0.2, 0) .. (0.5, 0.25, 1)
      .. controls +(0, +0.1, 0) and  +(0, -0.1, 0) .. (0.5, 0.4, 1)
      .. controls +(0, +0.1, 0) and  +(0, -0.1, 0) .. (0.7, 0.6, 1) coordinate (p)
      .. controls +(0, +0.1, 0) and  +(0, -0.1, 0) .. (0.5, 0.8, 1) -- (0.5,1,1);
      
      \draw[<-, very thin, blue] (0.7, 0.5, 0.9) .. controls +(0.5, 0, 0) and +(0,0, -0.3) .. (1, 0.5, 1.5) node[above] {$Q$};
      \draw (0.5, 0.4, 1)
      .. controls +(0, +0.1, 0) and  +(0, -0.1, 0) .. (0.3, 0.6, 1)
      .. controls +(0, +0.1, 0) and  +(0, -0.1, 0) .. (0.5, 0.8, 1);
      \draw (0.3, 0, 1)
      .. controls +(0, +0.2, 0) and  +(0, -0.2, 0) .. (0.5, 0.25, 1);

      \draw[very thick] (0.5,0,0.5) .. controls +(0,0.2, 0) and +(0, 0, -0.3) .. (0.5, 0.25,1);
      \draw[very thick] (0.5, 0.4, 1)
      .. controls +(0, 0, -0.30) and  +(0, -0.2, 0) .. (0.5, 0.6, 0.6)
      .. controls +(0, 0.2, 0)   and  +(0, 0, -0.3) .. (0.5, 0.8, 1); 

      \begin{scope}[gray]
        \draw (0.5, 0, 0) -- (0.5, 1, 0);
        
        \draw (0.5,1,0) -- (0.5,1,1);
     \end{scope}
     
    \end{scope}

  \end{scope}
\end{scope}


%% file: hhf_sample3.tex
\begin{scope}[scale=2]
  \tdplotsetmaincoords{80}{70}
  \begin{scope}[tdplot_main_coords]
          \draw[opacity=0, <-, very thin, blue] (0.7, 0.5, 0.9) .. controls +(0.5, 0, 0) and +(0,0, -0.3) .. (1, 0.5, 1.5) node[above] {$Q$};
          
    \coordinate (OOO) at (0, 0, 0);
    \coordinate (IOO) at (1, 0, 0);
    \coordinate (OIO) at (0, 1, 0);
    \coordinate (IIO) at (1, 1, 0);
    \coordinate (OOI) at (0, 0, 1);
    \coordinate (IOI) at (1, 0, 1);
    \coordinate (OII) at (0, 1, 1);
    \coordinate (III) at (1, 1, 1);
    \coordinate (sf) at (+3.5,  0.5,  0.5);
    \coordinate (sh) at (-2.5,  0.5,  0.5);
    \coordinate (sr) at ( 0.5, +2.25,  0.5);
    \coordinate (sl) at ( 0.5, -1.25,  0.5);
    \coordinate (st) at ( 0.5,  0.5, +2.25);
    \coordinate (sb) at ( 0.5,  0.5, -1.25);
    \begin{scope}[very thin]
      \draw (OOO) -- (OOI);
      \draw [densely dotted] (OOO) -- (OIO);
      \draw (OOO) -- (IOO);
      \draw (III) -- (OII);
      \draw (III) -- (IOI);
      \draw (III) -- (IIO);
      \draw (OOI) -- (OII);
      \draw (OOI) -- (IOI);
      \draw [densely dotted] (OIO) -- (OII);
      \draw [densely dotted] (OIO) -- (IIO);
      \draw (IOO) -- (IOI);
      \draw (IOO) -- (IIO);
    \end{scope}
    \begin{scope}
      \draw (0.7,0,0) -- (0.7, 0, 1);
      \draw (0.3,0,0) -- (0.3, 0, 1);
      \draw (0.7, 0, 1)
      .. controls +(0, +0.2, 0) and  +(0, -0.2, 0) .. (0.5, 0.25, 1)
      .. controls +(0, +0.1, 0) and  +(0, -0.1, 0) .. (0.5, 0.4, 1)
      .. controls +(0, +0.1, 0) and  +(0, -0.1, 0) .. (0.5, 0.8, 1) -- (0.5,1,1);
      \draw (0.3, 0, 1)
      .. controls +(0, +0.2, 0) and  +(0, -0.2, 0) .. (0.5, 0.25, 1);

      \draw[very thick] (0.5,0.25,0) -- (0.5, 0.25,1);


      \newcommand{\myop}{0}





      \begin{scope}[gray]

      \draw (0.7, 0, 0)
      .. controls +(0, +0.2, 0) and  +(0, -0.2, 0) .. (0.5, 0.25, 0)
      .. controls +(0, +0.1, 0) and  +(0, -0.1, 0) .. (0.5, 0.4, 0)
      .. controls +(0, +0.1, 0) and  +(0, -0.1, 0) .. (0.7, 0.6, 0)
      .. controls +(0, +0.1, 0) and  +(0, -0.1, 0) .. (0.5, 0.8, 0) -- (0.5,1,0);
      \draw (0.5, 0.4, 0)
      .. controls +(0, +0.1, 0) and  +(0, -0.1, 0) .. (0.3, 0.6, 0)
      .. controls +(0, +0.1, 0) and  +(0, -0.1, 0) .. (0.5, 0.8, 0);
      \draw (0.3, 0, 0)
      .. controls +(0, +0.2, 0) and  +(0, -0.2, 0) .. (0.5, 0.25, 0);

      \draw[black, very thick] (0.5, 0.4, 0)
      .. controls +(0, 0, 0.30) and  +(0, -0.2, 0) .. (0.5, 0.6, 0.4)
      .. controls +(0, 0.2, 0)   and  +(0, 0, 0.3) .. (0.5, 0.8, 0);

        \draw (0.5,1,0) -- (0.5,1,1);
     \end{scope}
     
    \end{scope}

  \end{scope}
\end{scope}


%% file: hhf_sample5.tex
\begin{scope}[scale=2]
  \tdplotsetmaincoords{80}{70}
  \begin{scope}[tdplot_main_coords]
      \draw[<-, very thin, blue] (0.3, 0.7, 0.5) .. controls +(-0.5, 0, 0) and +(0,0, -0.3) .. (-1, 0.6, 1.1) node[above] {$P$};

    \coordinate (OOO) at (0, 0, 0);
    \coordinate (IOO) at (1, 0, 0);
    \coordinate (OIO) at (0, 1, 0);
    \coordinate (IIO) at (1, 1, 0);
    \coordinate (OOI) at (0, 0, 1);
    \coordinate (IOI) at (1, 0, 1);
    \coordinate (OII) at (0, 1, 1);
    \coordinate (III) at (1, 1, 1);
    \coordinate (sf) at (+3.5,  0.5,  0.5);
    \coordinate (sh) at (-2.5,  0.5,  0.5);
    \coordinate (sr) at ( 0.5, +2.25,  0.5);
    \coordinate (sl) at ( 0.5, -1.25,  0.5);
    \coordinate (st) at ( 0.5,  0.5, +2.25);
    \coordinate (sb) at ( 0.5,  0.5, -1.25);
    \begin{scope}[very thin]
      \draw (OOO) -- (OOI);
      \draw [densely dotted] (OOO) -- (OIO);
      \draw (OOO) -- (IOO);
      \draw (III) -- (OII);
      \draw (III) -- (IOI);
      \draw (III) -- (IIO);
      \draw (OOI) -- (OII);
      \draw (OOI) -- (IOI);
      \draw [densely dotted] (OIO) -- (OII);
      \draw [densely dotted] (OIO) -- (IIO);
      \draw (IOO) -- (IOI);
      \draw (IOO) -- (IIO);
    \end{scope}
    \begin{scope}
      \draw (0.5, 0, 0) .. controls +(0, 0, 0.2) and +(0, 0, -0.2) .. (0.5, 0, 0.5)
      .. controls +(0, 0, 0.2) and +(0, 0, -0.2) .. (0.7, 0, 1);
      \draw (0.5, 0, 0.5)
      .. controls +(0, 0, 0.2) and +(0, 0, -0.2) .. (0.3, 0, 1);
      \draw (0.7, 0, 1)
      .. controls +(0, +0.2, 0) and  +(0, -0.2, 0) .. (0.5, 0.25, 1)
      .. controls +(0, +0.1, 0) and  +(0, -0.1, 0) .. (0.5, 0.4, 1)
      .. controls +(0, +0.1, 0) and  +(0, -0.1, 0) .. (0.5, 0.8, 1) -- (0.5,1,1);
      \draw (0.3, 0, 1)
      .. controls +(0, +0.2, 0) and  +(0, -0.2, 0) .. (0.5, 0.25, 1);

      \draw[very thick] (0.5,0,0.5) .. controls +(0,0.2, 0) and +(0, 0, -0.3) .. (0.5, 0.25,1);
      \draw[very thick] (0.5, 0.4, 0.5)
      .. controls +(0, 0, -0.15) and  +(0, -0.2, 0) .. (0.5, 0.6, 0.3)
      .. controls +(0, 0.2, 0)   and  +(0, 0, -0.15) .. (0.5, 0.8, 0.5)
      .. controls +(0, 0, 0.15) and  +(0, 0.2, 0) .. (0.5, 0.6, 0.7)
      .. controls +(0, -0.2, 0)   and  +(0, 0, 0.15) .. cycle;
      \draw[very thin, gray] (0.5, 0.4, 0.5)
      .. controls +(0, +0.1, 0) and  +(0, -0.1, 0) .. (0.7, 0.6, 0.5)
      .. controls +(0, +0.1, 0) and  +(0, -0.1, 0) .. (0.5, 0.8, 0.5)
      .. controls +(0, -0.1, 0) and  +(0, +0.1, 0) .. (0.3, 0.6, 0.5)
      .. controls +(0, -0.1, 0) and  +(0, +0.1, 0) .. cycle;

            \draw[<-, very thin, blue] (0.7, 0.5, 0.5) .. controls +(0.5, 0, 0) and +(0,0, -0.3) .. (1, 0.5, 1.5) node[above] {$Q$};

      \begin{scope}[gray]
        \draw (0.5, 0, 0) -- (0.5, 1, 0);
        
        \draw (0.5,1,0) -- (0.5,1,1);
     \end{scope}
     
    \end{scope}

  \end{scope}
\end{scope}


%% file: hhf_sample4.tex
\begin{scope}[scale=2]
  \tdplotsetmaincoords{80}{70}
  \begin{scope}[tdplot_main_coords]
    \coordinate (OOO) at (0, 0, 0);
    \coordinate (IOO) at (1, 0, 0);
    \coordinate (OIO) at (0, 1, 0);
    \coordinate (IIO) at (1, 1, 0);
    \coordinate (OOI) at (0, 0, 1);
    \coordinate (IOI) at (1, 0, 1);
    \coordinate (OII) at (0, 1, 1);
    \coordinate (III) at (1, 1, 1);
    \coordinate (sf) at (+3.5,  0.5,  0.5);
    \coordinate (sh) at (-2.5,  0.5,  0.5);
    \coordinate (sr) at ( 0.5, +2.25,  0.5);
    \coordinate (sl) at ( 0.5, -1.25,  0.5);
    \coordinate (st) at ( 0.5,  0.5, +2.25);
    \coordinate (sb) at ( 0.5,  0.5, -1.25);
    \begin{scope}[very thin]
      \draw (OOO) -- (OOI);
      \draw [densely dotted] (OOO) -- (OIO);
      \draw (OOO) -- (IOO);
      \draw (III) -- (OII);
      \draw (III) -- (IOI);
      \draw (III) -- (IIO);
      \draw (OOI) -- (OII);
      \draw (OOI) -- (IOI);
      \draw [densely dotted] (OIO) -- (OII);
      \draw [densely dotted] (OIO) -- (IIO);
      \draw (IOO) -- (IOI);
      \draw (IOO) -- (IIO);
    \end{scope}
    \begin{scope}
      \draw (0.5, 0, 0) .. controls +(0, 0, 0.2) and +(0, 0, -0.2) .. (0.5, 0, 0.5)
      .. controls +(0, 0, 0.2) and +(0, 0, -0.2) .. (0.7, 0, 1);
      \draw (0.5, 0, 0.5)
      .. controls +(0, 0, 0.2) and +(0, 0, -0.2) .. (0.3, 0, 1);
      \draw (0.7, 0, 1)
      .. controls +(0, +0.2, 0) and  +(0, -0.2, 0) .. (0.5, 0.25, 1)
      .. controls +(0, +0.1, 0) and  +(0, -0.1, 0) .. (0.5, 0.4, 1)
      .. controls +(0, +0.1, 0) and  +(0, -0.1, 0) .. (0.5, 0.8, 1) -- (0.5,1,1);
      \draw (0.3, 0, 1)
      .. controls +(0, +0.2, 0) and  +(0, -0.2, 0) .. (0.5, 0.25, 1);

      \draw[very thick] (0.5,0,0.5) .. controls +(0,0.2, 0) and +(0, 0, -0.3) .. (0.5, 0.25,1);

            \draw[<-, very thin, blue] (0.7, 0.5, 0.5) .. controls +(0.5, 0, 0) and +(0,0, -0.3) .. (1, 0.5, 1.5) node[above] {$R$};
      \begin{scope}[gray]
        \draw (0.5, 0, 0) -- (0.5, 1, 0);
        
        \draw (0.5,1,0) -- (0.5,1,1);
     \end{scope}
     
    \end{scope}

  \end{scope}
\end{scope}


%% file: hhf_sample00.tex
\begin{scope}[scale=2]
  \tdplotsetmaincoords{0}{0}
  \begin{scope}[tdplot_main_coords]
      \draw[>->] (0.7, 0, 1)
      .. controls +(0, +0.2, 0) and  +(0, -0.2, 0) .. (0.5, 0.25, 1)
      .. controls +(0, +0.1, 0) and  +(0, -0.1, 0) .. (0.5, 0.4, 1)
      .. controls +(0, +0.1, 0) and  +(0, -0.1, 0) .. (0.5, 0.6, 1) coordinate (q) node {$\bullet$}
      .. controls +(0, +0.1, 0) and  +(0, -0.1, 0) .. (0.5, 0.8, 1) -- (0.5,1,1);
      \draw[>-] (0.3, 0, 1)
      .. controls +(0, +0.2, 0) and  +(0, -0.2, 0) .. (0.5, 0.25, 1);
      \node[right] at (q) {$R$};

  \end{scope}
\end{scope}


%% file: sym_MOYvertex.tex
\begin{scope}
\draw[->] (0,0) -- (0,0.5) node [at end, above] {$a+b$};  
\draw[>-] (-0.5, -0.5) .. controls +(0,0) and +(0, -0.2) .. (0,0) node [at start, below] {$a$};  
\draw[>-] (+0.5, -0.5) .. controls +(0,0) and +(0, -0.2) .. (0,0) node [at start, below] {$b$};  
\begin{scope}[xshift = 5cm]
  \draw[-<] (0,0) -- (0,0.5) node [at end, above] {$a+b$};  
\draw[<-] (-0.5, -0.5) .. controls +(0,0) and +(0, -0.2) .. (0,0) node [at start, below] {$a$};  
\draw[<-] (+0.5, -0.5) .. controls +(0,0) and +(0, -0.2) .. (0,0) node [at start, below] {$b$};  
\end{scope}
\end{scope}

%% file: hhf_bl-web-example.tex
\begin{scope}[font= \tiny, scale= 1.8]
  \draw[very thin] (0,0) rectangle (2,2);
  \draw[->-] (0.3, 0) -- +(0, 0.1) node[pos=0, below] {$4$};
  \draw[->-] (1, 0) -- +(0, 0.1) node[pos=0, below] {$5$};
  \draw[->-] (1.7, 0) -- +(0, 0.1) node[pos=0, below] {$2$};
  \draw[->] (0.25, 1.9) -- +(0, 0.1) node[pos=1, above] {$2$};
  \draw[->] (0.75, 1.9) -- +(0, 0.1) node[pos=1, above] {$4$};
  \draw[->] (1.25, 1.9) -- +(0, 0.1) node[pos=1, above] {$2$};
  \draw[->] (1.75, 1.9) -- +(0, 0.1) node[pos=1, above] {$3$};

  \draw (0.3, 0.1)  coordinate (a1) .. controls + (0,0.2) and +(0, -0.2) ..
  (0.65, 0.6) coordinate (b2) node [below, pos=0.5] {$2$} .. controls + (0,0.2) and +(0, -0.2) ..
  (0.4, 1.1) coordinate (c1) node [above, pos=0.3] {$3$}.. controls + (0,0.2) and +(0, -0.2) ..
  (0.35, 1.5) coordinate (d1) node [left, pos=0.5] {$5$} .. controls + (0,0.2) and +(0, -0.2) ..
  (0.25, 1.9) coordinate (e1); 
  \draw (a1) .. controls + (0,0.2) and +(0, -0.2) .. (c1) node [left, pos=0.5] {$2$};
  \draw
  (1, 0.1) coordinate (a2)  .. controls + (0,0.2) and +(0, -0.2) .. (b2) node [below, pos=0.5] {$1$};
  \draw (d1) .. controls + (0,0.2) and +(0, -0.2) .. (0.75, 1.9) coordinate (e2) node [above, pos=0.5] {$3$};
  \draw (a2) -- +(0, 0.6) coordinate[pos =0.7] (b3) node[right, pos=0.5] {$4$} .. controls + (0,0.2) and +(0, -0.2) .. (e2) node [right, pos=0.5] {$1$};
  \draw (b3) .. controls + (0,0.2) and +(0, -0.2) .. (1.4, 1) coordinate (c3) node[pos =0.5, above] {$3$} -- +(0,0.5) coordinate (d3) node [right, pos=0.5] {$5$} .. controls + (0,0.2) and +(0, -0.2) .. (1.25, 1.9);
  \draw (d3) .. controls + (0,0.2) and +(0, -0.2) .. (1.75, 1.9);
  \draw (1.7,0.1) .. controls + (0,0.2) and +(0, -0.2) .. (c3);
  

\end{scope}

%% file: cef_3localmodels.tex
\begin{scope}
\tdplotsetmaincoords{80}{140}
  \begin{scope}[tdplot_main_coords]
    \filldraw [very thin, fill=red, opacity = 0.2] (0,-1, -1) -- (0,1,-1) -- (0,1,1) -- (0,-1,1) -- (0,-1,-1);
    \node[sloped, red] at (0,0,0) {$a$};
  \end{scope}
\begin{scope}[xshift = 3cm, tdplot_main_coords]
  \begin{scope}
    \filldraw [very thin, fill =red, opacity = 0.2] (0,-1, -1) -- (0,1,-1) -- (0,1,0) -- (0,-1,0) -- (0,-1,-1);
        \node[sloped, red] at (0,0,-0.5) {$a+b$};
    \end{scope}
  \begin{scope}[rotate around y = 150]
    \filldraw [very thin, fill =blue, opacity = 0.2] (0,-1, -1) -- (0,1,-1) -- (0,1,0) -- (0,-1,0) -- (0,-1,-1);
        \node[sloped, blue] at (0,0,-0.5) {$a$};
  \end{scope}
  \begin{scope}[rotate around y =-130]
    \filldraw [very thin, fill =green, opacity = 0.2] (0,-1, -1) -- (0,1,-1) -- (0,1,0) -- (0,-1,0) -- (0,-1,-1);
        \node[sloped, green!50!black] at (0,0,-0.5) {$b$};
        \draw[very thick, ->] (0,1,0) -- (0,-1,0);
  \end{scope}
  \end{scope}
\begin{scope}[scale = 1.6, xshift = 4.5cm, tdplot_main_coords]
  \begin{scope}
    \filldraw [very thin, fill =red, opacity = 0.2] (0,-1, -1) -- (0,1,-1) -- (0,1,0) -- (0,0,0) -- (0,-1,-1);
    \coordinate (a) at (0, -1, -1);
        \node[sloped, red] at (0,0.2,-0.5) {$a+b+c$};
    \end{scope}
  \begin{scope}[rotate around y = 150]
    \filldraw [very thin, fill= blue, opacity = 0.2] (0,-1, -1) -- (0,1,-1) -- (0,1,0) -- (0,0,0) -- (0,-1,-1);
    \coordinate (b) at (0, -1, -1);
        \node[sloped, blue] at (0,0.5,-0.5) {$a+b$};
  \end{scope}
  \begin{scope}[rotate around y =-130]
    \filldraw [very thin, fill= green, opacity = 0.2] (0,-1, -1) -- (0,1,-1) -- (0,1,0) -- (0,0,0) -- (0,-1,-1);
    \coordinate (c) at (0, -1, -1);
        \node[sloped, green!50!black] at (0,0.5,-0.5) {$c$};
        \node[sloped, purple!50!black] at (0,1,-1.5) {$a$};
        \node[sloped, green!50!black] at (0,+0.5,-1.5) {$b$};
        \node[sloped, gray] at (0,-1.3,0.1) {$b+c$};
  \end{scope}
  \filldraw[very thin, fill= orange, opacity = 0.2] (a) -- (b) -- (0,0,0) -- (a);
  \filldraw[very thin, fill= gray, opacity = 0.2] (a) -- (c) -- (0,0,0) -- (a);
  \filldraw[very thin, fill= purple, opacity = 0.2] (b) -- (c) -- (0,0,0) -- (b);
  \draw[very thick, ->] (a) -- (0,0,0);
  \draw[very thick, <-] (b) -- (0,0,0);
  \draw[very thick, ->] (c) -- (0,0,0);
  \draw[very thick, <-] (0,1, 0) -- (0,0,0);
  \fill[green!50!black, opacity = 0.6] (0,0,0) circle (0.5mm);
\end{scope}
\end{scope}

%% file: sw_lhrule.tex
\begin{scope}
 \draw (0,0) -- +(0:1);
 \draw (0,0) -- +(120:1);
 \draw (0,0) -- +(240:1);
 \filldraw[fill= white, draw=black, very thin] (0,0) circle (0.15cm);
 \filldraw[fill = black] (0,0) circle (0.02cm);
 \draw[very thin,->] (-10:0.8cm) arc (-10:-110 :0.8); 
 \draw[very thin,->] (230:0.8cm) arc (230: 130:0.8); 
 \draw[very thin,->] (110:0.8cm) arc (110:10 :0.8); 
\end{scope}

%% file: sym_cube.tex
\tdplotsetmaincoords{80}{70}
\begin{scope}[tdplot_main_coords]
  \coordinate (OOO) at (0, 0, 0);
  \coordinate (IOO) at (1, 0, 0);
  \coordinate (OIO) at (0, 1, 0);
  \coordinate (IIO) at (1, 1, 0);
  \coordinate (OOI) at (0, 0, 1);
  \coordinate (IOI) at (1, 0, 1);
  \coordinate (OII) at (0, 1, 1);
  \coordinate (III) at (1, 1, 1);
  \coordinate (sf) at (+3.5,  0.5,  0.5);
  \coordinate (sh) at (-2.5,  0.5,  0.5);
  \coordinate (sr) at ( 0.5, +2.25,  0.5);
  \coordinate (sl) at ( 0.5, -1.25,  0.5);
  \coordinate (st) at ( 0.5,  0.5, +2.25);
  \coordinate (sb) at ( 0.5,  0.5, -1.25);
  \draw (OOO) -- (OOI);
  \draw [densely dotted] (OOO) -- (OIO);
  \draw (OOO) -- (IOO);
  \draw (III) -- (OII);
  \draw (III) -- (IOI);
  \draw (III) -- (IIO);
  \draw (OOI) -- (OII);
  \draw (OOI) -- (IOI);
  \draw [densely dotted] (OIO) -- (OII);
  \draw [densely dotted] (OIO) -- (IIO);
  \draw (IOO) -- (IOI);
  \draw (IOO) -- (IIO);
  \draw[<-] ($(OOO)!0.5!(OII)$) -- +(-2,0,0) ;
  \draw[<-] ($(OOO)!0.5!(IOI)$) -- +(0,-1,0) ;
  \draw[<-] ($(OOO)!0.5!(IIO)$) -- +(0,0,-1) ;
  \draw[<-] ($(III)!0.5!(IOO)$) -- +(+2,0,0) ;
  \draw[<-] ($(III)!0.5!(OIO)$) -- +(0,+1,0) ;
  \draw[<-] ($(III)!0.5!(OOI)$) -- +(0,0,+1) ;
  \node at (sf) {$s_f$};
  \node at (sh) {$s_h$};
  \node at (sr) {$s_r$};
  \node at (sl) {$s_l$};
  \node at (st) {$s_t$};
  \node at (sb) {$s_b$};
  \begin{scope}[xshift =2cm, yshift =1.5cm]
    \draw[->] (0,0,0) -- (-1,0,0);
    \node at (-1.5,0,0) {$x_2$};
    \draw[->] (0,0,0) -- (0,0.5,0);
    \node at (0,0.75,0) {$x_1$};
    \draw[->] (0,0,0) -- (0,0,0.5);
    \node at (0,0,.75) {$x_3$};
  \end{scope}

\end{scope}

%% file: hhf_dot-migration-foam.tex
\begin{scope}
  \begin{scope}
    \draw (0:0) -- +(90:1) coordinate[pos=0.5] (t) ;
    \draw (0:0) .. controls +(90:-0.3) and +(0,0) .. (-120:1) coordinate[pos=0.5] (a)  node[pos=1, below] {$a$};
    \draw (0:0) .. controls +(90:-0.3) and +(0,0) .. (-60:1) coordinate[pos=0.5] (b)  node[pos=1, below] {$b$};
    \begin{scope}[xshift= 1.5cm, yshift = 0.5cm ]
      \draw (0:0) -- +(90:1) coordinate[pos=0.5] (t);
      \draw (0:0) .. controls +(90:-0.3) and +(0,0) .. (-120:1) coordinate[pos=0.5] (a);
      \draw (0:0) .. controls +(90:-0.3) and +(0,0) .. (-60:1) coordinate[pos=0.5] (b);
  \end{scope}
  \draw [thick] (0:0) -- +(1.5, 0.5);
  \draw  (90:1) -- +(1.5, 0.5)  node [pos=0.5, above, sloped] {$a+b$};
  \draw (-60:1) -- +(1.5, 0.5);
  \draw (-120:1) -- +(1.5, 0.5);
  \coordinate[xshift = 0.5cm, yshift= 0.3cm] (a) at (-120:1){}; 
  \coordinate[xshift = 1cm, yshift= 0.6cm] (b) at (-60:1){}; 
  \coordinate (ab) at (0.75,0.75);
  \draw[<-, very thin] (ab) .. controls +(-0.3, -0.2) and +(0,0) .. +(-1, -0.3) node [left] {$P$}; 
\end{scope}
   \node at (3.7, 0.25) {$\displaystyle{= \quad\sum_{i}}$};

    \begin{scope} [xshift = 5.5cm]
    \draw (0:0) -- +(90:1) coordinate[pos=0.5] (t) ;
    \draw (0:0) .. controls +(90:-0.3) and +(0,0) .. (-120:1) coordinate[pos=0.5] (a)  node[pos=1, below] {$a$};
    \draw (0:0) .. controls +(90:-0.3) and +(0,0) .. (-60:1) coordinate[pos=0.5] (b)  node[pos=1, below] {$b$};
    \begin{scope}[xshift= 1.5cm, yshift = 0.5cm ]
      \draw (0:0) -- +(90:1) coordinate[pos=0.5] (t);
      \draw (0:0) .. controls +(90:-0.3) and +(0,0) .. (-120:1) coordinate[pos=0.5] (a);
      \draw (0:0) .. controls +(90:-0.3) and +(0,0) .. (-60:1) coordinate[pos=0.5] (b);
  \end{scope}
  \draw [thick] (0:0) -- +(1.5, 0.5);
  \draw  (90:1) -- +(1.5, 0.5)  node [pos=0.5, above, sloped] {$a+b$};
  \draw (-60:1) -- +(1.5, 0.5);
  \draw (-120:1) -- +(1.5, 0.5);
  \coordinate[xshift = 0.5cm, yshift= 0.3cm] (a) at (-120:1){}; 
  \coordinate[xshift = 1cm, yshift= 0.6cm] (b) at (-60:1){}; 
  \coordinate (ab) at (0.75,0.75);
  \draw[<-, very thin] (a) .. controls +(-0.2, 0.2) and +(0.2,0) .. +(-0.5, 0.5) node [left] {$Q_i^{(1)}$};
  \draw[<-, very thin] (b) .. controls +( 0.2, 0.2) and +(-0.2,0) .. +(0.5, 0.5) node [right] {$Q_i^{(2)}$}; 

\end{scope}

\end{scope}


%% file: hhf_bubble-2.tex
\begin{scope}[scale=0.7]
\begin{scope}
  \draw (1, 0.5) .. controls +(1, -1) and +(-1, -1) .. (4, 0.5) coordinate[pos=0.5, yshift = 0.1cm] (bot);
  \filldraw[fill = white, fill opacity =0.7] (0,0) -- (1,1) -- (5,1) -- (4,0) -- cycle;
  \filldraw[fill = white, fill opacity =0.7] (1, 0.5) .. controls +(1, +1) and +(-1, +1) .. (4, 0.5) coordinate[pos=0.5, yshift = -0.1cm] (top);
  \draw (2.5, 0.5) circle (1.5cm and 0.2cm);
  \node at (2.5, -0.7) {$G$};
  \draw [<-, very thin] (bot) .. controls +(0, -0.3) and +(0.3, 0) .. +(-1.5, -0.6) node [left] {$Q\in \QQ[x_1, \dots x_b]^{\sg[b]}$};
  \draw [<-, very thin] (top) .. controls +(0, 0.3) and +(0.3, 0) .. +(-1.5, 0.6) node [left] {$P \in \QQ[x_1, \dots x_a]^{\sg[a]}$};
  \coordinate (a) at (5, 0.5);
\end{scope}

\begin{scope}[xshift = 7cm]
  \filldraw[fill = white, fill opacity =0.7] (0,0) -- (1,1) -- (5,1) -- (4,0) -- cycle;
    \draw [<-, very thin] (1.5,0.5) .. controls +(0, +0.5) and +(-1, 0) .. +(-1, 1.7) node [right, yshift= -0.2 cm] {$\displaystyle{R= \sum_{\substack{I\sqcup J = \{1, \dots, a+b\} \\ \#I = a, \#J =b}} \frac{P(x_I)Q(x_J)}{\prod_{\substack{i \in I \\ j \in J}}(x_i - x_j)}}$};

    \node at (2.5, -0.7) {$G'$};
    \coordinate (b) at (0, 0.5);
\end{scope}
\node at (6, 0.5) {$=$};

\end{scope}

%% file: hhf_2-sided-dish-2.tex
\begin{scope}[scale=0.7]
\begin{scope}
  \filldraw[fill = white, fill opacity =0.7, yshift =-0.7cm] (0,0) -- (1,1) -- (5,1) -- (4,0) -- cycle;
  \filldraw[fill = white, fill opacity =0.7, yshift =0.7 cm] (0,0) -- (1,1) -- (5,1) -- (4,0) -- cycle;
  \draw (2.5, 0.5) circle (1.5cm and 0.1cm);
  \draw[very thin] (4, 0.5) .. controls +(0.3, 0.3) and +(-0.5, -0.5) .. (4.5, 1.2);
  \draw[very thin] (4, 0.5) .. controls +(0.3, -0.3) and +(-0.3, -0.3) .. (4.5, -0.2);
  \draw[very thin] (1, 0.5) .. controls +(-0.3, 0.3) and +(0.2, 0.2) .. (0.5, 1.2);
  \draw[very thin] (1, 0.5) .. controls +(-0.3, -0.3) and +(0.3, 0.3) .. (0.5, -0.2);

  \coordinate (a) at (5, 0.5);
\end{scope}

\begin{scope}[xshift = 9cm]
  \filldraw[fill = white, fill opacity =0.7, yshift = -0.7cm] (0,0) -- (1,1) -- (5,1) -- (4,0) -- cycle;
  \filldraw[fill = white, fill opacity =0.7, yshift = +0.7cm] (0,0) -- (1,1) -- (5,1) -- (4,0) -- cycle;
   \draw [<-, very thin] (2.5,1.2) .. controls +(0, +0.5) and +(-0.3, 0) .. +(1.5, 1 ) node [right] {$\displaystyle{s_\lambda}$};
   \draw [<-, very thin] (2.5,-0.2) .. controls +(0, -0.5) and +(-0.3, 0) .. +(1.5, -1 ) node [right] {$\displaystyle{s_{\widehat{\lambda}}}$};
   \node (c) at (-1.5, 0.5) {$\displaystyle{\sum_{\lambda \in T(a,b)}} (-1)^{|\widehat{\lambda}|}$};
    \coordinate (b) at (-3, 0.5);
\end{scope}
\node at (6, 0.5) {$=$};

\end{scope}

%% file: hhf_assoc.tex
\begin{scope}
  \begin{scope}
    \draw[->] (0,0) .. controls +(0, 0.2) and +(0, -0.2) ..
    (-0.25, 0.5) .. controls +(0, 0.2) and +(0, -0.2) ..
    (-.5,1) node[above] {$a$};
    \draw[->] (0,0) .. controls +(0, 0.2) and +(0, -0.2) ..
    ( .5,1) node[above] {$c$};
    \draw[->] 
    (-0.25, 0.5) .. controls +(0, 0.2) and +(0, -0.2) ..
    ( 0 ,1) node[above] {$b$};
    \draw[-<] (0,0) -- (0, -0.25) node[below] {$a+b+c$};
  \end{scope}
  \node at (1.3, 0.25) {$\displaystyle{\simeq}$};
    \begin{scope}[xshift = 2.6cm]
    \draw[->] (0,0) .. controls +(0, 0.2) and +(0, -0.2) ..
    (-.5,1) node[above] {$a$};
    \draw[->] (0,0) .. controls +(0, 0.2) and +(0, -0.2) ..
    (0.25, 0.5) .. controls +(0, 0.2) and +(0, -0.2) ..
    ( .5,1) node[above] {$c$};
    \draw[->] 
    (0.25, 0.5) .. controls +(0, 0.2) and +(0, -0.2) ..
    ( 0 ,1) node[above] {$b$};
    \draw[-<] (0,0) -- (0, -0.25) node[below] {$a+b+c$};      
    \end{scope}
\end{scope}


%% file: hhf_co-assoc.tex
\begin{scope}[yscale=-1]
  \begin{scope}
    \draw[-<] (0,0) .. controls +(0, 0.2) and +(0, -0.2) ..
    (-0.25, 0.5) .. controls +(0, 0.2) and +(0, -0.2) ..
    (-.5,1) node[below] {$a$};
    \draw[-<] (0,0) .. controls +(0, 0.2) and +(0, -0.2) ..
    ( .5,1) node[below] {$c$};
    \draw[-<] 
    (-0.25, 0.5) .. controls +(0, 0.2) and +(0, -0.2) ..
    ( 0 ,1) node[below] {$b$};
    \draw[->] (0,0) -- (0, -0.25) node[above] {$a+b+c$};
  \end{scope}
  \node at (1.3, 0.25) {$\displaystyle{\simeq}$};
    \begin{scope}[xshift = 2.6cm]
    \draw[-<] (0,0) .. controls +(0, 0.2) and +(0, -0.2) ..
    (-.5,1) node[below] {$a$};
    \draw[-<] (0,0) .. controls +(0, 0.2) and +(0, -0.2) ..
    (0.25, 0.5) .. controls +(0, 0.2) and +(0, -0.2) ..
    ( .5,1) node[below] {$c$};
    \draw[-<] 
    (0.25, 0.5) .. controls +(0, 0.2) and +(0, -0.2) ..
    ( 0 ,1) node[below] {$b$};
    \draw[->] (0,0) -- (0, -0.25) node[above] {$a+b+c$};      
    \end{scope}
\end{scope}


%% file: hhf_digon.tex
\begin{scope}
  \begin{scope}
    \draw[>->, ->-] (0,0) --
    (0, 0.25) .. controls+ (0, 0.2) and +(0, -0.2) ..
    (-0.25, 0.75) node [left] {$a$} .. controls+ (0, 0.2) and +(0, -0.2) ..
    (0, 1.25) -- (0,1.5) node[above] {$a+b$};
    \draw[ ->-] 
    (0, 0.25) .. controls+ (0, 0.2) and +(0, -0.2) ..
    (0.25, 0.75) node[right] {$b$} .. controls+ (0, 0.2) and +(0, -0.2) ..
    (0, 1.25); 
  \end{scope}
  \node at (1.5, 0.75) {$\displaystyle{\simeq \,\, \begin{bmatrix} a+b \\ a \end{bmatrix}}$};
  \begin{scope}[xshift = 2.8cm]
    \draw[ ->] (0,0) -- (0, 1.5) node [above] {$a+b$};
    \end{scope}
\end{scope}


%% file: hhf_square-1.tex
\begin{scope}
  \begin{scope}
    \draw[>->, ->-] (0,0) -- +(0,2) node [pos= 0, below] {$a$};
    \draw[>->, ->-] (1,0) -- +(0,2) node [pos= 0, below] {$b$};
    \draw[->-] (1, 0.3) .. controls + (0, 0.2) and +(0, -0.2) .. (0, 0.7) node[pos=0.5, below] {$c$};
    \draw[->-] (0, 1.3) .. controls + (0, 0.2) and +(0, -0.2) .. (1, 1.7) node[pos=0.5, above] {$d$};
  \end{scope}
  \node at (3, 0.75) {$\displaystyle{\simeq \,\,\bigoplus_{i=0}^d \begin{bmatrix} b-a +d-c \\ d-i
                                                         \end{bmatrix}}$};
  \begin{scope}[xshift = 5cm]
    \draw[>->, ->-] (0,0) -- +(0,2) node [pos= 0, below] {$a$};
    \draw[>->, ->-] (1,0) -- +(0,2) node [pos= 0, below] {$b$};
    \draw[->-] (0, 0.3) .. controls + (0, 0.2) and +(0, -0.2) .. (1, 0.7) node[pos=0.5, below] {$i$};
    \draw[->-] (1, 1.3) .. controls + (0, 0.2) and +(0, -0.2) .. (0, 1.7) node[pos=0.5, above, font =\tiny,  sloped] {$c+i-d$};
  \end{scope}
\end{scope}


%% file: hhf_square-2.tex
\begin{scope}
  \begin{scope}
    \draw[>->, ->-] (0,0) -- +(0,2) node [pos= 0, below] {$a$};
    \draw[>->, ->-] (1,0) -- +(0,2) node [pos= 0, below] {$b$};
    \draw[->-] (0, 0.3) .. controls + (0, 0.2) and +(0, -0.2) .. (1, 0.7) node[pos=0.5, below] {$c$};
    \draw[->-] (1, 1.3) .. controls + (0, 0.2) and +(0, -0.2) .. (0, 1.7) node[pos=0.5, above] {$d$};
  \end{scope}
  \node at (3, 0.75) {$\displaystyle{\simeq \,\,\bigoplus_{i=0}^d \begin{bmatrix} a-b +d-c \\ d-i
                                                         \end{bmatrix}}$};
  \begin{scope}[xshift = 5cm]
    \draw[>->, ->-] (0,0) -- +(0,2) node [pos= 0, below] {$a$};
    \draw[>->, ->-] (1,0) -- +(0,2) node [pos= 0, below] {$b$};
    \draw[->-] (1, 0.3) .. controls + (0, 0.2) and +(0, -0.2) .. (0, 0.7) node[pos=0.5, below] {$i$};
    \draw[->-] (0, 1.3) .. controls + (0, 0.2) and +(0, -0.2) .. (1, 1.7) node[pos=0.5, above, font =\tiny, sloped] {$c+i-d$};
  \end{scope}
\end{scope}


%% file: hhf_annular-foam-example.tex
\begin{scope}[scale=1]
  \tdplotsetmaincoords{80}{70}
  \begin{scope}[tdplot_main_coords]
  \begin{scope} [canvas is xy plane at z=0]
    \begin{scope}[very thin]
      \draw (0,0) circle (3cm);
      \draw (0,0) circle (1cm);
      \coordinate (a) at (70:3cm);
      \coordinate (aa) at (250:3cm);
    \end{scope}

    \begin{scope}
      \draw (0,0) circle (2.2);
      \draw (0,0) circle (1.7);
      \draw (60:2.7) arc (60: 180: 2.7cm);
      \draw (240:2.7) arc (240: 360: 2.7cm);
      \draw (180: 1.2) arc (180: 420: 1.2cm);
      \draw (180:2.7) .. controls +(270: 0.3) and
      +(100:0.3) .. (200:2.2);
      \draw (220:2.2) .. controls +(310: 0.3) and
      +(150:0.3) .. (240:2.7);
      \draw (0:2.7) .. controls +(90: 0.3) and
      +(-70:0.3) .. (20:2.2);
      \draw (40:2.2) .. controls +(130: 0.3) and
      +(-30:0.3) .. (60:2.7);
      \draw (60:1.2) .. controls +(150: 0.3) and
      +(-10:0.3) .. (80:1.7);
      \draw (160:1.7) .. controls +(250: 0.3) and
      +(90:0.3) .. (180:1.2);
      \draw[rotate=20] (30:1.7) .. controls +(120:0.3) and
      +(-40: 0.3) .. (50:2.2);
      \draw[rotate=90] (30:1.7) .. controls +(120:0.3) and
      +(-40: 0.3) .. (50:2.2);
      \draw[rotate=270] (30:1.7) .. controls +(120:0.3) and
      +(-40: 0.3) .. (50:2.2);
      \draw[rotate=50] (30:2.2) .. controls +(120:0.3) and
      +(-40: 0.3) .. (50:1.7);
      \draw[rotate=230] (30:2.2) .. controls +(120:0.3) and
      +(-40: 0.3) .. (50:1.7);
      \coordinate (A) at (110:1);
      \coordinate (B) at (110:3);
      \coordinate (E) at (210:1);
      \coordinate (F) at (210:3);
      \coordinate (B1) at (110: 1.2);
      \coordinate (B2) at (110: 1.7);
      \coordinate (B3) at (110: 2.2);
      \coordinate (B4) at (110: 2.7);
      \coordinate (C1) at (210: 1.2);
      \coordinate (C2) at (210: 1.7);
      \coordinate (C3) at (210: 2.2);
      \coordinate (C4) at (210: 2.7);
    \end{scope}

  \end{scope}
    \begin{scope} [canvas is xy plane at z=2]
    \begin{scope}[very thin]
      \coordinate (T1) at (110: 1.2);
      \coordinate (T2) at (110: 1.7);
      \coordinate (T3) at (110: 2.2);
      \coordinate (T4) at (110: 2.7);

      \coordinate (S1) at (210: 1.2);
      \coordinate (S2) at (210: 1.7);
      \coordinate (S3) at (210: 2.2);
      \coordinate (S4) at (210: 2.7);
      \coordinate (b) at (70:3cm);
      \coordinate (bb) at (250:3cm);

      \draw (0,0) circle (3cm);
      \draw (0,0) circle (1cm);
      \draw (a) -- (b);
      \draw (aa) -- (bb);
      \coordinate (D) at (110:1);
      \coordinate (C) at (110:3); 
      \fill[opacity = 0.2, red] (A) -- (B) -- (C) -- (D);
      \coordinate (H) at (210:1);
      \coordinate (G) at (210:3); 
      \fill[opacity = 0.2, green!50!black] (E) -- (F) -- (G) -- (H);
    \end{scope}
  \end{scope}
  \begin{scope}[draw= red!70!black, very thick]
  \draw (B2) -- +(0,0, 0.3) coordinate (c1) .. controls +(0,0,0.2) and + (0,0,-0.2) .. +(0,0,1.2) coordinate (c2)  +(0,0,1.5)  .. controls +(0,0,0.2) and + (0,0,-0.2) ..  (T1);
  \draw (B4) -- (T4) coordinate [pos=0.2] (G1) coordinate [pos =0.5] (G2) coordinate [pos= 0.7] (G3);
  \draw (B3) -- (T3) coordinate [pos=0.2] (F1) coordinate [pos =0.5] (F2) coordinate [pos= 0.6] (F3);
  \draw (T2) -- + (0,0, -0.2) coordinate (D3) .. controls +(0,0,-0.2) and + (0,0,0.2) .. (c1) coordinate[pos =0.7] (D2) coordinate[pos =0.3] (D1);
  \draw (G1) .. controls +(0,0,0.2) and + (0,0,-0.2) .. (F2);
  \draw (F1) .. controls +(0,0,0.2) and + (0,0,-0.2) .. (D2);
  \draw (F3) .. controls +(0,0,0.2) and + (0,0,-0.2) .. (G3);
\end{scope}
  \begin{scope}[draw= green!50!black, very thick]
    \draw (S2) -- +(0,0, -0.8) coordinate (i) coordinate[pos=0.5] (h);
    \draw (C1) -- +(0,0,0.8) coordinate[pos=0.7] (m) .. controls +(0,0,0.2) and + (0,0,-0.2) .. (i);
    \draw (C3) -- +(0,0,1) coordinate (j) coordinate[pos=0.5] (k) .. controls +(0,0,0.2) and + (0,0,-0.2) .. (S4);
    \draw (C2) -- +(0,0,0.2) coordinate (l) .. controls +(0,0,0.2) and + (0,0,-0.2) .. (k);
    \draw (l) .. controls +(0,0,0.2) and + (0,0,-0.2) .. (m);
    \draw (C3) --(S3);
\end{scope}

    \begin{scope} [canvas is xy plane at z=2]

    \begin{scope}[rotate = 180, xscale = -1]
            \draw (0,0) circle (2.2);
      \draw (0,0) circle (1.7);
      \draw (60:2.7) arc (60: 180: 2.7cm);
      \draw (240:2.7) arc (240: 360: 2.7cm);
      \draw (180: 1.2) arc (180: 420: 1.2cm);
      \draw (180:2.7) .. controls +(270: 0.3) and
      +(100:0.3) .. (200:2.2);
      \draw (220:2.2) .. controls +(310: 0.3) and
      +(150:0.3) .. (240:2.7);
      \draw (0:2.7) .. controls +(90: 0.3) and
      +(-70:0.3) .. (20:2.2);
      \draw (40:2.2) .. controls +(130: 0.3) and
      +(-30:0.3) .. (60:2.7);
      \draw (60:1.2) .. controls +(150: 0.3) and
      +(-10:0.3) .. (80:1.7);
      \draw (160:1.7) .. controls +(250: 0.3) and
      +(90:0.3) .. (180:1.2);
      \draw[rotate=20] (30:1.7) .. controls +(120:0.3) and
      +(-40: 0.3) .. (50:2.2);
      \draw[rotate=90] (30:1.7) .. controls +(120:0.3) and
      +(-40: 0.3) .. (50:2.2);
      \draw[rotate=270] (30:1.7) .. controls +(120:0.3) and
      +(-40: 0.3) .. (50:2.2);
      \draw[rotate=50] (30:2.2) .. controls +(120:0.3) and
      +(-40: 0.3) .. (50:1.7);
      \draw[rotate=230] (30:2.2) .. controls +(120:0.3) and
      +(-40: 0.3) .. (50:1.7);
    \end{scope}
  \end{scope}
\end{scope}
\end{scope}

%% file: hhf_dot-migration-web.tex
\begin{scope}
  \begin{scope}
    \draw[->] (0:0) -- +(90:1) coordinate[pos=0.5] (t)  node [pos=1, above] {$a+b$};
    \draw[-<] (0:0) .. controls +(90:-0.3) and +(0,0) .. (-120:1) coordinate[pos=0.5] (a)  node[pos=1, below] {$a$};
    \draw[-<] (0:0) .. controls +(90:-0.3) and +(0,0) .. (-60:1) coordinate[pos=0.5] (b)  node[pos=1, below] {$b$};
    \node at (t) {$\bullet$};
    \node[right] at (t) {$P$};
  \end{scope}
  \node at (1.5, 0) {$\displaystyle{= \,\sum_{i}}$};
    \begin{scope}[xshift = 3cm]
    \draw[->] (0:0) -- +(90:1) coordinate[pos=0.5] (t)  node [pos=1, above] {$a+b$};
    \draw[-<] (0:0) .. controls +(90:-0.3) and +(0,0) .. (-120:1) coordinate[pos=0.5] (a)  node[pos=1, below] {$a$};
    \draw[-<] (0:0) .. controls +(90:-0.3) and +(0,0) .. (-60:1) coordinate[pos=0.5] (b)  node[pos=1, below] {$b$};
    \node at (a) {$\bullet$};
    \node[left] at (a) {$Q^{(1)}_i$};
    \node at (b) {$\bullet$};
    \node[right] at (b) {$Q^{(2)}_i$};
  \end{scope}
\end{scope}


%% file: hhf_dot-migration-web-2.tex
\begin{scope}[yscale=-1]
  \begin{scope}
    \draw[-<] (0:0) -- +(90:1) coordinate[pos=0.5] (t)  node [pos=1, below] {$a+b$};
    \draw[->] (0:0) .. controls +(90:-0.3) and +(0,0) .. (-120:1) coordinate[pos=0.5] (a)  node[pos=1, above] {$a$};
    \draw[->] (0:0) .. controls +(90:-0.3) and +(0,0) .. (-60:1) coordinate[pos=0.5] (b)  node[pos=1, above] {$b$};
    \node at (t) {$\bullet$};
    \node[right] at (t) {$P$};
  \end{scope}
  \node at (1.5, 0) {$\displaystyle{= \,\sum_{i}}$};
    \begin{scope}[xshift = 3cm]
    \draw[-<] (0:0) -- +(90:1) coordinate[pos=0.5] (t)  node [pos=1, below] {$a+b$};
    \draw[->] (0:0) .. controls +(90:-0.3) and +(0,0) .. (-120:1) coordinate[pos=0.5] (a)  node[pos=1, above] {$a$};
    \draw[->] (0:0) .. controls +(90:-0.3) and +(0,0) .. (-60:1) coordinate[pos=0.5] (b)  node[pos=1, above] {$b$};
    \node at (a) {$\bullet$};
    \node[left] at (a) {$Q^{(1)}_i$};
    \node at (b) {$\bullet$};
    \node[right] at (b) {$Q^{(2)}_i$};
  \end{scope}
\end{scope}


%% file: hhf_std-tree2.tex
\begin{scope}[scale= 0.5, font=\tiny]
  \draw[gray] (-2.5, 5) rectangle (2.5, 0);
  \draw[>->] (0,0) --(0,1) node[pos=0, below] {$\myk$} .. controls +(0, 0.5) and +(0, -0.5) ..  (-2,5)  node[pos=1, above] {$k_1$};
  \draw[->] (0,1) .. controls +(0, 0.5) and +(0, -0.5) .. (0.5,2) .. controls +(0, 0.5) and +(0, -0.5) .. (1.5,4) .. controls +(0, 0.5) and +(0, -0.5) .. (2,5)  node[pos=1, above] {$k_\ell$}; 
  \draw[->] (1.5,4) .. controls +(0, 0.5) and +(0, -0.5) .. (1,5) node[pos=1, above]  {$k_{\ell-1}$};
  \node[above] at (0, 5) {$\cdots$};
  \draw[->] (0.5,2) .. controls +(0, 0.5) and +(0, -0.5) .. (-1,5) node[pos=1, above] {$k_2$};
\end{scope}

%% file: hhf_bl-foam-example.tex
\begin{scope}[scale=2]
  \tdplotsetmaincoords{80}{70}
  \begin{scope}[tdplot_main_coords]
    \coordinate (OOO) at (0, 0, 0);
    \coordinate (IOO) at (1, 0, 0);
    \coordinate (OIO) at (0, 1, 0);
    \coordinate (IIO) at (1, 1, 0);
    \coordinate (OOI) at (0, 0, 1);
    \coordinate (IOI) at (1, 0, 1);
    \coordinate (OII) at (0, 1, 1);
    \coordinate (III) at (1, 1, 1);
    \coordinate (sf) at (+3.5,  0.5,  0.5);
    \coordinate (sh) at (-2.5,  0.5,  0.5);
    \coordinate (sr) at ( 0.5, +2.25,  0.5);
    \coordinate (sl) at ( 0.5, -1.25,  0.5);
    \coordinate (st) at ( 0.5,  0.5, +2.25);
    \coordinate (sb) at ( 0.5,  0.5, -1.25);
    \begin{scope}[very thin]
      \draw (OOO) -- (OOI);
      \draw [densely dotted] (OOO) -- (OIO);
      \draw (OOO) -- (IOO);
      \draw (III) -- (OII);
      \draw (III) -- (IOI);
      \draw (III) -- (IIO);
      \draw (OOI) -- (OII);
      \draw (OOI) -- (IOI);
      \draw [densely dotted] (OIO) -- (OII);
      \draw [densely dotted] (OIO) -- (IIO);
      \draw (IOO) -- (IOI);
      \draw (IOO) -- (IIO);
    \end{scope}
    \begin{scope}
      \draw (0.5, 0, 0) .. controls +(0, 0, 0.2) and +(0, 0, -0.2) .. (0.5, 0, 0.3)
      .. controls +(0, 0, 0.2) and +(0, 0, -0.2) .. (0.7, 0, 0.7)
      .. controls +(0, 0, 0.2) and +(0, 0, -0.2) .. (0.9, 0, 1);
      \draw (0.5, 0, 0.3)
      .. controls +(0, 0, 0.2) and +(0, 0, -0.2) .. (0.1, 0, 1);
       \draw (0.7, 0, 0.7)
      .. controls +(0, 0, 0.2) and +(0, 0, -0.2) .. (0.5, 0, 1);
      \draw (0.5, 0, 1)
      .. controls +(0, +0.2, 0) and  +(0, -0.2, 0) .. (0.3, 0.3, 1)
      .. controls +(0, +0.2, 0) and  +(0, -0.2, 0) .. (0.5, 0.7, 1)
      .. controls +(0, +0.2, 0) and  +(0, -0.2, 0) .. (0.9, 0.9, 1) -- (0.9,1,1);
      \draw (0.1, 0, 1)
      .. controls +(0, +0.2, 0) and  +(0, -0.2, 0) .. (0.3, 0.3, 1);

      \draw (0.9, 0,1) -- (0.9, 1,1);
      \draw (0.9, 0.2, 1)
      .. controls +(0, +0.2, 0) and  +(0, -0.2, 0) .. (0.5, 0.7, 1);

      \draw (0.3, 0.3, 1)  .. controls +(0, +0.2, 0) and  +(0, -0.2, 0) .. (0.1, 1, 1);
      \draw (0.5, 0.7, 1)  .. controls +(0, +0.2, 0) and  +(0, -0.2, 0) .. (0.5, 1, 1);
      \begin{scope}[gray]
      \draw (0.5, 0, 0) -- (0.5, 1, 0);
      \draw (0.5, 1, 0) .. controls +(0, 0, 0.2) and +(0, 0, -0.2) .. (0.5, 1, 0.3)
      .. controls +(0, 0, 0.2) and +(0, 0, -0.2) .. (0.7, 1, 0.7)
      .. controls +(0, 0, 0.2) and +(0, 0, -0.2) .. (0.9, 1, 1);
      \draw (0.5, 1, 0.3)
      .. controls +(0, 0, 0.2) and +(0, 0, -0.2) .. (0.1, 1, 1);
       \draw (0.7, 1, 0.7)
       .. controls +(0, 0, 0.2) and +(0, 0, -0.2) .. (0.5, 1, 1);
     \end{scope}
     
    \end{scope}

  \end{scope}
\end{scope}


%% file: hhf_stackingFE-2.tex
\begin{scope}[scale=1.4]

\begin{scope}
\draw[densely dashed, ->] (2.7, 0.5) -- +(1.1, 0) node[midway, above, font=\tiny] {stack};
\draw[densely dashed, ->] (5.2, 0.5) -- +(1.1, 0) node[midway, above, font=\tiny] {deform};
\end{scope}
  \begin{scope}
  \filldraw[fill=red!50!white] (0,0) rectangle +(1,1) node[midway, yshift=-1cm] {$F$};
  \filldraw[fill=green!50!white] (1.5,0) rectangle +(1,1) node[midway, yshift= -1cm] {$E$};
\end{scope}
\begin{scope}[xshift = 4cm]
  \filldraw[fill=red!50!white] (0,0.5) rectangle +(1,0.5);
  \filldraw[fill=green!50!white] (0,00) rectangle +(1,0.5);
\end{scope}
\begin{scope}[xshift = 6.5cm]
  \filldraw[fill=red!50!white] (0,1) -- ++(1,0) arc (0:-180:0.5); 
  \filldraw[fill=green!50!white] (0,1) -- ++(0,-1) -- ++(1,0) -- ++(0,1) arc (0:-180:0.5) node[midway, yshift= -1cm] {$F(E)$};
\end{scope}

\end{scope}

%% file: hhf_defoamation.tex
\begin{scope}[scale=2]
\begin{scope}
  \filldraw[fill=green!50!white] (0,0) rectangle +(1,1);
  \foreach \i in{0.1, 0.2, 0.3, 0.4}{
      \draw[densely dotted] (\i, 0) -- +(0,1);
      \draw[densely dashed, xshift= 0.5cm] (\i, 0) -- +(0,1);
    }
    \node[below] at (0.25, 0) {$\web_0$};
    \node[below] at (0.75, 0) {$\web_1$};
    \node[above] at (0.25, 1) {$\web_2$};
    \node[above] at (0.75, 1) {$\web_3$};
  \end{scope}

\begin{scope}[xshift= 2cm]
  \filldraw[fill=green!50!white] (0,0) rectangle +(1,1);
  \foreach \i in{0.1, 0.2, 0.3, 0.4}{
      \draw[densely dotted, xshift=0.5cm] (0:\i cm) arc (0:180:\i cm) +(0,1);
      \draw[densely dashed, yshift= 1cm, xshift=0.5cm] (0:\i cm) arc (0:-180:\i cm);
    }
    \node[below] at (0.25, 0) {$\web_2^\dagger$};
    \node[below] at (0.75, 0) {$\web_0$};
    \node[above] at (0.25, 1) {$\web_3$};
    \node[above] at (0.75, 1) {$\web_1^\dagger$};
  \end{scope}

\end{scope}


%% file: hhf_bubble-1.tex
\begin{scope}[scale=0.7]
\begin{scope}
  \draw (1, 0.5) .. controls +(1, -1) and +(-1, -1) .. (4, 0.5) coordinate[pos=0.5, yshift = 0.1cm] (bot);
  \filldraw[fill = white, fill opacity =0.7] (0,0) -- (1,1) -- (5,1) -- (4,0) -- cycle;
  \filldraw[fill = white, fill opacity =0.7] (1, 0.5) .. controls +(1, +1) and +(-1, +1) .. (4, 0.5) coordinate[pos=0.5, yshift = -0.1cm] (top);
  \draw (2.5, 0.5) circle (1.5cm and 0.2cm);
  \node at (2.5, -0.7) {$G$};
  \draw [<-, very thin] (bot) .. controls +(0, -0.3) and +(0.3, 0) .. +(-1.5, -0.6) node [left] {$Q\in \QQ[x_1, \dots x_b]^{\sg[b]}$};
  \draw [<-, very thin] (top) .. controls +(0, 0.3) and +(0.3, 0) .. +(-1.5, 0.6) node [left] {$P \in \QQ[x_1, \dots x_a]^{\sg[a]}$};
  \coordinate (a) at (5, 0.5);
\end{scope}

\begin{scope}[xshift = 8cm]
  \filldraw[fill = white, fill opacity =0.7] (0,0) -- (1,1) -- (5,1) -- (4,0) -- cycle;
    \draw [<-, very thin] (1.5,0.5) .. controls +(0, +0.5) and +(-1, 0) .. +(-1, 1.7) node [right, yshift= -0.2 cm] {$\displaystyle{R= \sum_{\substack{I\sqcup J = \{1, \dots, a+b\} \\ \#I = a, \#J =b}} \frac{P(x_I)Q(x_J)}{\prod_{\substack{i \in I \\ j \in J}}(x_i - x_j)}}$};
    \node at (2.5, -0.7) {$G'$};
    \coordinate (b) at (0, 0.5);
\end{scope}
\draw[<->, densely dashed] (a) -- (b) node[above, font =\small, pos=0.5] {$\infty$-equivalence};

\end{scope}

%% file: hhf_2-sided-dish-1.tex
\begin{scope}[scale=0.7]
\begin{scope}
  \filldraw[fill = white, fill opacity =0.7, yshift =-0.7cm] (0,0) -- (1,1) -- (5,1) -- (4,0) -- cycle;
  \filldraw[fill = white, fill opacity =0.7, yshift =0.7 cm] (0,0) -- (1,1) -- (5,1) -- (4,0) -- cycle;
  \draw (2.5, 0.5) circle (1.5cm and 0.1cm);
  \draw[very thin] (4, 0.5) .. controls +(0.3, 0.3) and +(-0.5, -0.5) .. (4.5, 1.2);
  \draw[very thin] (4, 0.5) .. controls +(0.3, -0.3) and +(-0.3, -0.3) .. (4.5, -0.2);
  \draw[very thin] (1, 0.5) .. controls +(-0.3, 0.3) and +(0.2, 0.2) .. (0.5, 1.2);
  \draw[very thin] (1, 0.5) .. controls +(-0.3, -0.3) and +(0.3, 0.3) .. (0.5, -0.2);

  \node at (2.5, -1.2) {$G$};
  \coordinate (a) at (5, 0.5);
\end{scope}

\begin{scope}[xshift = 11cm]
  \filldraw[fill = white, fill opacity =0.7, yshift = -0.7cm] (0,0) -- (1,1) -- (5,1) -- (4,0) -- cycle;
  \filldraw[fill = white, fill opacity =0.7, yshift = +0.7cm] (0,0) -- (1,1) -- (5,1) -- (4,0) -- cycle;
   \draw [<-, very thin] (2.5,1.2) .. controls +(0, +0.5) and +(-0.3, 0) .. +(1.5, 1 ) node [right] {$\displaystyle{s_\lambda}$};
   \draw [<-, very thin] (2.5,-0.2) .. controls +(0, -0.5) and +(-0.3, 0) .. +(1.5, -1 ) node [right] {$\displaystyle{s_{\widehat{\lambda}}}$};
   \node (c) at (-1.5, 0.5) {$\displaystyle{\sum_{\lambda \in T(a,b)}} (-1)^{|\widehat{\lambda}|}$};
    \coordinate (b) at (-3, 0.5);
\end{scope}
\draw[<->, densely dashed] (a) -- (b) node[above, font =\small, pos=0.5] {$\infty$-equivalence};

\end{scope}